\documentclass[a4paper,11pt]{article}

\usepackage{authblk}
\usepackage{parskip}
\usepackage{mathtools}
\usepackage{amssymb}
\usepackage{amsmath}
\usepackage{latexsym}
\usepackage{mathrsfs}  
\usepackage{bbm}       
\usepackage{amsthm}    
\usepackage{relsize}   
\usepackage{graphicx}
\usepackage{thmtools}  
\usepackage{mathrsfs}  
\usepackage{paralist}  
\usepackage{lipsum}    
\usepackage[all]{xy}   

\numberwithin{equation}{section}

\usepackage{titlesec}   

\titleformat{\section}[block]{\bfseries\filcenter}
{{\upshape\thesection\enspace}}{.5em}{}

\titleformat{\subsection}[block]{\filcenter}
{{\upshape\thesubsection\enspace}}{.5em}{} 

\usepackage{enumitem}  
\setlist{nosep}  
\setitemize[0]{leftmargin=*}
\setenumerate[0]{leftmargin=*}
\setenumerate[1]{label={(\arabic*)}} 

\newcommand{\N}{\mathbb{N}}     
\newcommand{\Q}{\mathbb{Q}}     
\newcommand{\R}{\mathbb{R}}     
\newcommand{\C}{\mathbb{C}}     
\newcommand{\Prob}{\mathbb{P}}  
\newcommand{\Exp}{\mathbb{E}}   
\newcommand{\st}{\,:\,}         
\newcommand{\norm}[1]{\left|\left|#1\right|\right|}              
\newcommand{\triplet}[3]{\left( #1, #2, #3 \right) }             
\newcommand{\ProbSpace}{\triplet{\Omega}{\mathscr{F}}{\Prob}}    
\newcommand{\abs}[1]{\left| #1 \right|}                          
\renewcommand{\qedsymbol}{$\square$}                       
\newcommand{\Levy}{L\'{e}vy} 
\newcommand{\cadlag}{c\`{a}dl\`{a}g }
\newcommand{\Ito}{It\^{o}}
\newcommand{\defeq}{\mathrel{\mathop:}=}                         
\newcommand\restr[2]{{
  \left.\kern-\nulldelimiterspace 
  #1 
  \vphantom{\big|} 
  \right|_{#2} 
  }}
  
\theoremstyle{plain} 
\newtheorem{theo}{Theorem}[section]    
\newtheorem{prop}[theo]{Proposition} 
\newtheorem{coro}[theo]{Corollary}
\newtheorem{lemm}[theo]{Lemma}
\newtheorem{assu}[theo]{Assumption}
\newtheorem{rema}[theo]{Remark}

\theoremstyle{definition} 
\newtheorem{defi}[theo]{Definition}
\newtheorem{exam}[theo]{Example}

\declaretheoremstyle[%
  spaceabove=-5pt,%
  spacebelow=6pt,%
  headfont=\normalfont\itshape,%
  postheadspace=1em,%
  qed=\qedsymbol%
]{mystyle} 
\declaretheorem[name={Proof},style=mystyle,unnumbered,
]{prf}

 \title{\Large Existence of Continuous and C\`{a}dl\`{a}g Versions for Cylindrical Processes in the Dual of a Nuclear Space}
\author{C. A. Fonseca-Mora}
\affil{\small\small School of Mathematics and Statistics,  The University of Sheffield, Hicks Building, \\
Hounsfield Road, S3 7RH, United Kingdom, \\

and \\

Escuela de Matem\'{a}tica, Universidad de Costa Rica, San Pedro de Montes de Oca, \\
San Jos\'{e}, Apartado Postal 11501-2060, Costa Rica. \\
\smallskip

Email: christianandres.fonseca@ucr.ac.cr
}
\date{}

\begin{document}



 \maketitle

\abstract{Let $\Phi$ be a nuclear space and let $\Phi'_{\beta}$ denote its strong dual. In this paper we introduce sufficient conditions for a cylindrical process in $\Phi'$ to have a version that is a $\Phi'_{\beta}$-valued continuous or c\'{a}dl\'{a}g process. We also establish sufficient conditions for the existence of such a version taking values and having finite moments in a Hilbert space continuously embedded in $\Phi'_{\beta}$. Finally, we apply our results to the study of properties of cylindrical martingales in $\Phi'$.}

\smallskip



\section{Introduction}

Probability theory in infinite dimensional spaces has been an area under intensive development since the 1960s. There are several monographs devoted to the study of stochastic processes and measures defined on these spaces; we can cite for example  \cite{Badrikian, DaleckyFomin,  Heyer, Ito, LedouxTalagrand, Linde, SchwartzRM, VakhaniaTarieladzeChobanyan}. 

Among the range of topics within the theory of probability on infinite dimensional spaces is the study of cylindrical probability measures, that is, finitely additive set functions that have ``projections'' on finite dimensional spaces that are probability measures. In general, cylindrical probability measures do not necessarily extend to a countably additive probability measures; the investigation of this problem led to the celebrated discoveries of Minlos (see \cite{Minlos:1963}) and Sazonov (see \cite{Sazonov:1958}) in the context of duals of countably Hilbertian nuclear spaces and of Hilbert spaces respectively, which establish necessary and sufficient conditions in terms of the continuity of the Fourier transform of a cylindrical measure for this extension to be possible. 

Associated to the concept of cylindrical measures is that of cylindrical random variables. These are similar to the usual random variables but their laws are cylindrical, rather than bona fide, probability measures. 
Recently, cylindrical processes, i.e. collections of cylindrical random variables indexed by time, are increasingly attracting attention because they appear naturally as the driving noise of stochastic partial differential equations (see e.g. \cite{ApplebaumRiedle:2010, KallianpurXiong, PeszatZabczyk, PriolaZabczyk:2011, Riedle:2014, Riedle:2015, vanNeervenVeraarWeis:2008}).

Parallel to the problem of  finding conditions for the extension of cylindrical probability measures to countably additive probability measures is the problem of finding a random variable $Y$ in the underlying space that is a version of a given  cylindrical random variable $X$, where here ``version'' means that $Y$ coincides almost surely with $X$ when evaluated with respect to any element of the dual space. In the case of the dual of a nuclear space, the regularization theorem of \Ito{} and Nawata (see \cite{ItoNawata:1983}) established necessary and sufficient conditions for the existence of such a version for a cylindrical random variable. As a consequence of the regularization theorem we can establish sufficient conditions for the existence of an stochastic process that is a version of a given cylindrical process. 

In \cite{Mitoma:1983}, Mitoma addressed the more specialized problem of establishing sufficient conditions for an stochastic process taking values in the dual of a nuclear Fr\'{e}chet space to have a continuous or a right-continuous with left limits version. There are several applications for this result. Some examples are the study of self-intersection local times of density processes in $\mathcal{S'}(\R^{d})$ (see \cite{BojdeckiTalarczyk:2005}) and the study of regularity properties of the limit process of many-server queueing systems (see \cite{KaspiRamanan:2013}). By using different techniques, several authors extended Mitoma's work by establishing different conditions on the nuclear space and on the stochastic processes (see \cite{Fernique:1989, Fouque:1984, Martias:1988}). 

In this paper we carried out a further extension by showing that under some natural conditions a cylindrical processes in the dual of a nuclear space (no other conditions are assumed on the space) has a continuous or a right-continuous with left limits version. To the extend of our knowledge our result is the more general result of its type that can be found on the literature. Moreover, our result is a natural generalization of the regularization theorem of \Ito{} and Nawata for cylindrical random variables.  

As well as from proving this result, we also introduce some new results on the existence of continuous and c\`{a}dl\`{a}g versions taking values in a Hilbert space continuously embedded in the dual of a nuclear space. Then we apply our results to the existence of continuous and c\`{a}dl\`{a}g versions for stochastic processes and to the study of martingales taking values in the dual of a nuclear space. 

We remark that we have successfully applied the results obtained in this paper to other problems, for example to prove the \Levy -\Ito{} decomposition for L\'{e}vy processes taking values in the strong dual of a reflexive nuclear space and to the introduction of a new theory of stochastic integration for operator-valued processes with respect to L\'{e}vy processes on the dual of a nuclear space; both of these are furthermore applied to the study of stochastic evolution equations driven by L\'{e}vy noise in duals of nuclear spaces (see \cite{FonsecaMoraThesis}). These results will be published elsewhere. 

The organization of the paper is as follows. In Section \ref{sectionPrelim} we review the basic concepts on nuclear spaces and of random processes defined on its strong dual that we will need throughout this paper. Section \ref{sectionRegulTheorem} is devoted to the proof of our main result (Theorem \ref{theoRegularizationTheoremCadlagContinuousVersion}). In Section \ref{sectionContCadVersionHilbertSpace} we prove several modifications of our result for the existence of continuous and c\`{a}dl\`{a}g versions taking values in a Hilbert space continuously embedded in the dual of a nuclear space. Finally, in Section \ref{sectionApplicaStochProcesses} we apply our results to study cylindrical martingales in the dual of a nuclear space.

\section{Preliminaries} \label{sectionPrelim}

\subsection{Nuclear Spaces} \label{subsectionNuclSpace}

In this section we introduce our notation and review some of the key concepts that we will need throughout this paper. For detailed information on locally convex spaces the reader is referred to \cite{Jarchow, Schaefer, Treves}.  

Let $\Phi$ be a locally convex space (over $\R$ or $\C$). If $p$ is a continuous semi-norm on $\Phi$ and $r>0$, the closed ball of radius $r$ of $p$ given by $B_{p}(r) = \left\{ \phi \in \Phi: p(\phi) \leq r \right\}$ is a closed, convex, balanced neighborhood of zero in $\Phi$. 

A continuous semi-norm (respectively a norm) $p$ on $\Phi$ is called \emph{Hilbertian} if $p(\phi)^{2}=Q(\phi,\phi)$, for all $\phi \in \Phi$, where $Q$ is a symmetric, non-negative bilinear form (respectively inner product) on $\Phi \times \Phi$. Let $\Phi_{p}$ be the Hilbert space that corresponds to the completion of the pre-Hilbert space $(\Phi / \mbox{ker}(p), \tilde{p})$, where $\tilde{p}(\phi+\mbox{ker}(p))=p(\phi)$ for each $\phi \in \Phi$. The quotient map $\Phi \rightarrow \Phi / \mbox{ker}(p)$ has an unique continuous linear extension $i_{p}:\Phi \rightarrow \Phi_{p}$.  

Let $q$ be another continuous Hilbertian semi-norm on $\Phi$ for which $p \leq q$. In this case, $\mbox{ker}(q) \subseteq \mbox{ker}(p)$. Moreover, the inclusion map from $\Phi / \mbox{ker}(q)$ into $\Phi / \mbox{ker}(p)$ is linear and continuous, and therefore it has a unique continuous extension $i_{p,q}:\Phi_{q} \rightarrow \Phi_{p}$. Furthermore, we have the following relation: $i_{p}=i_{p,q} \circ i_{q}$. 

We denote by $\Phi'$ the topological dual of $\Phi$ and by $f[\phi]$ the canonical pairing of elements $f \in \Phi'$, $\phi \in \Phi$. We denote by $\Phi'_{\beta}$ the dual space $\Phi'$ equipped with its \emph{strong topology} $\beta$, i.e. $\beta$ is the topology on $\Phi'$ generated by the family of semi-norms $\{ \eta_{B} \}$, where for each $B \subseteq \Phi'$ bounded we have $\eta_{B}(f)=\sup \{ \abs{f[\phi]}: \phi \in B \}$ for all $f \in \Phi'$.  If $p$ is a continuous Hilbertian semi-norm on $\Phi$, then we denote by $\Phi'_{p}$ the Hilbert space dual to $\Phi_{p}$. The dual norm $p'$ on $\Phi'_{p}$ is given by $p'(f)=\sup \{ \abs{f[\phi]}:  \phi \in E, \, p(\phi) \leq 1 \}$ for all $ f \in \Phi'_{p}$. Moreover, the dual operator $i_{p}'$ corresponds to the canonical inclusion from $\Phi'_{p}$ into $\Phi'_{\beta}$ and it is linear and continuous. 

Let $p$ and $q$ be continuous Hilbertian semi-norms on $\Phi$ such that $p \leq q$.
The space of continuous linear operators (respectively Hilbert-Schmidt operators) from $\Phi_{q}$ into $\Phi_{p}$ is denoted by $\mathcal{L}(\Phi_{q},\Phi_{p})$ (respectively $\mathcal{L}_{2}(\Phi_{q},\Phi_{p})$) and the operator norm (respectively Hilbert-Schmidt norm) is denote by $\norm{\cdot}_{\mathcal{L}(\Phi_{q},\Phi_{p})}$ (respectively $\norm{\cdot}_{\mathcal{L}_{2}(\Phi_{q},\Phi_{p})}$). We employ an analogue notation for operators between the dual spaces $\Phi'_{p}$ and $\Phi'_{q}$. 

Among the many equivalent definitions of a nuclear space (see \cite{Pietsch, Treves}),  the following is the most useful for our purposes. 

\begin{defi}
A (Hausdorff) locally convex space $(\Phi,\mathcal{T})$ is called \emph{nuclear} if its topology $\mathcal{T}$ is generated by a family $\Pi$ of Hilbertian semi-norms such that for each $p \in \Pi$ there exists $q \in \Pi$, satisfying $p \leq q$ and the canonical inclusion $i_{p,q}: \Phi_{q} \rightarrow \Phi_{p}$ is Hilbert-Schmidt. 
\end{defi}

\begin{exam} 
The following locally convex spaces are nuclear (see \cite{Treves}, Chapter 51 and \cite{Pietsch}, Chapter 6): the (Schwartz) space of rapidly decreasing functions $\mathscr{S}(\R^{d})$ and the space of tempered distributions $\mathscr{S}'(\R^{d})$; the space $\mathscr{D}(X)$ of test functions on $X$ and the space of distributions $\mathscr{D}'(X)$  ($X$: open subset of $\R^{d}$); the space $\mathcal{C}^{\infty}(U)$ of infinitely differentiable real or complex functions on $U$ ($U$: open subset of $\R^{d}$) and the space $\mathcal{H}(G)$ of holomorphic functions on $G$ ($G$: open subset of $\C^{n}$).  
\end{exam}

\begin{defi}
Let $\Phi$ be a vector space and let $\Pi$ by a family of Hilbertian semi-norms on $\Phi$ such that for each $p \in \Pi$ the Hilbert space $\Phi_{p}$ is separable. The locally convex topology $\mathcal{T}$ on $\Phi$ generated by the family $\Pi$  is called \emph{multi-Hilbertian} (respectively \emph{countably Hilbertian} if $\Pi$ is countable). The space $(\Phi,\mathcal{T})$ is called a  \emph{multi-Hilbertian} space (respectively \emph{countably Hilbertian}).  
\end{defi}

If $(\Phi,\mathcal{T})$ is a nuclear space and $p$ is a continuous Hilbertian semi-norm on $\Phi$, the Hilbert space $\Phi_{p}$ is separable (see \cite{Pietsch}, Proposition 4.4.9 and Theorem 4.4.10, p.82). Therefore, every nuclear space is multi-Hilbertian.  

Let $(\Phi,\mathcal{T})$ be a multi-Hilbertian space and let $\{ p_{n} \}_{n \in \N}$ be an increasing sequence of continuous Hilbertian semi-norms on $(\Phi,\mathcal{T})$. Denote by $\theta$ the countably Hilbertian topology on $\Phi$ generated by the semi-norms 
$\{ p_{n} \}_{n \in \N}$. The topology $\theta$ is weaker than $\mathcal{T}$. We denote by $\Phi_{\theta}$ the space $(\Phi,\theta)$. Some properties of the space $\Phi_{\theta}$ are listed below:

\begin{prop} \label{propWeakerCountablyHilbertianTopology} In the above notation,  
\begin{enumerate}
\item $\Phi_{\theta}$ is a separable pseudo-metrizable locally convex space and $d:\Phi \rightarrow \R$ given by $d(\phi)=\sum_{n \in \N} 2^{-n} [p_{n}(\phi)/(1+p_{n}(\phi))]$ for $\phi \in \Phi$ is a pseudo-metric for $\Phi_{\theta}$.    
\item The completion $\widetilde{\Phi_{\theta}}$ of $\Phi_{\theta}$ is separable, complete, pseudo-metrizable space and it is furthermore a Baire space. 
\item 
\begin{equation} \label{dualOfWeakerCountablyHilbertianTopology}
(\widetilde{\Phi_{\theta}})'= \Phi'_{\theta}=\bigcup_{n \in \N} \Phi'_{p_{n}}. 
\end{equation}
\end{enumerate}
\end{prop}
\begin{prf}
\emph{(1)} The fact that $\Phi_{\theta}$ is pseudo-metrizable and that $d$ is a pseudo-metric is a consequence that its topology is generated by a countable number of semi-norms (see \cite{NariciBeckenstein}, Theorem 5.6.1, p.123). The fact that $\Phi_{\theta}$ is separable is a consequence of the fact that $\{ p_{n} \}_{n \in \N}$ generates the topology $\theta$ and that $\Phi_{p_{n}}$ is separable for each $n \in \N$ (see \cite{Pietsch}, Theorem 4.4.10, p.82-3). 

\emph{(2)} The completion $\widetilde{\Phi_{\theta}}$ of $\Phi_{\theta}$ is again a pseudo-metrizable locally convex space (see \cite{NariciBeckenstein}, Theorem 3.7.1, p.60) and hence by (an extension of) the Baire category theorem (see \cite{NariciBeckenstein}, Theorem 11.7.2, p.393) it follows that $\widetilde{\Phi_{\theta}}$ is a Baire space.  

\emph{(3)} It is well-know that the dual space of $\Phi_{\theta}$ is equal to $\bigcup_{n \in \N} \Phi'_{p_{n}}$. On the other hand, the dual space $(\widetilde{\Phi_{\theta}})'$ of $\widetilde{\Phi_{\theta}}$ can be identified (algebraically) with the dual space $\Phi'_{\theta}$ of $\Phi_{\theta}$ (see \cite{Jarchow}, Corollary 4, Section 3.4, p.63), thus we obtain \eqref{dualOfWeakerCountablyHilbertianTopology}. 
\end{prf} 

\begin{rema}
Even if $(\Phi,\mathcal{T})$ is Hausdorff this does not in general implies that $\Phi_{\theta}$ is Hausdorff. However, if $\Phi_{\theta}$ is Hausdorff the pseudo-metric $d$ defined in Proposition \ref{propWeakerCountablyHilbertianTopology} is a metric and hence $\Phi_{\theta}$ is metrizable (see \cite{Jarchow}, Theorem 1, Section 2.8, p.40). Moreover, in that case $\widetilde{\Phi_{\theta}}$ is a Fr\'{e}chet space and it is (isomorphic to) the projective limit $\mbox{proj}_{n \in \N} \, \Phi_{p_{n}}$ (see \cite{Jarchow}, Corollary 7, Section 3.4, p.63). 
\end{rema}

\subsection{Cylindrical and Stochastic Processes} \label{subSectionCylAndStocProcess}

\begin{assu}
Unless otherwise specified, in this section $\Phi$ will always denote a multi-Hilbertian space over $\R$.
\end{assu} 
 
Let $\ProbSpace$ be a complete probability space. We denote by $L^{0} \ProbSpace$ the space of equivalence classes of real-valued random variables defined on $\ProbSpace$. We always consider the space $L^{0} \ProbSpace$ equipped with the topology of convergence in probability and in this case it is a complete, metrizable, topological vector space, but is not in general locally convex (see e.g. \cite{Badrikian}). 

A Borel measure $\mu$ on $\Phi'_{\beta}$ is called a \emph{Radon measure} if for every $\Gamma \in \mathcal{B}(\Phi'_{\beta})$ and $\epsilon >0$, there exist a compact set $K_{\epsilon} \subseteq \Gamma$ such that $\mu(\Gamma \backslash K_{\epsilon}) < \epsilon$. In general not every Borel measure on $\Phi$ is Radon, however, when $\Phi$ is a Fr\'{e}chet nuclear space or a countable inductive limit of Fr\'{e}chet nuclear spaces, then every Borel measure on $\Phi'_{\beta}$ is a Radon measure (see Corollary 1.3 of Dalecky and Fomin \cite{DaleckyFomin}, p.11).

For any $n \in \N$ and any $\phi_{1}, \dots, \phi_{n} \in \Phi$, we define a linear map $\pi_{\phi_{1}, \dots, \phi_{n}}: \Phi' \rightarrow \R^{n}$ by
\begin{equation} \label{defiMapProjectionCylinder}
\pi_{\phi_{1}, \dots, \phi_{n}}(f)=(f[\phi_{1}], \dots, f[\phi_{n}]), \quad \forall \, f \in \Phi'. 
\end{equation}
The map $\pi_{\phi_{1}, \dots, \phi_{n}}$ is clearly linear and continuous. Let $M$ be a subset of $\Phi$. A subset of $\Phi'$ of the form 
\begin{equation} \label{defiCylindricalSet}
\mathcal{Z}\left(\phi_{1}, \dots, \phi_{n}; A \right) = \left\{ f \in \Phi'\st \left(f[\phi_{1}], \dots, f[\phi_{n}]\right) \in A \right\}= \pi_{\phi_{1}, \dots, \phi_{n}}^{-1}(A)
\end{equation}
where $n \in \N$, $\phi_{1}, \dots, \phi_{n} \in M$ and $A \in \mathcal{B}\left(\R^{n}\right)$ is called a \emph{cylindrical set} based on $M$. The set of all the cylindrical sets based on $M$ is denoted by $\mathcal{Z}(\Phi',M)$. It is an algebra but if $M$ is a finite set then it is a $\sigma$-algebra. The $\sigma$-algebra generated by $\mathcal{Z}(\Phi',M)$ is denoted by $\mathcal{C}(\Phi',M)$ and it is called the \emph{cylindrical $\sigma$-algebra} with respect to $(\Phi',M)$. 
If $M=\Phi$, we write $\mathcal{Z}(\Phi')=\mathcal{Z}(\Phi',\Phi)$ and $\mathcal{C}(\Phi')=\mathcal{C}(\Phi',\Phi)$\index{.c CPhi@$\mathcal{C}(\Phi')$}. One can easily see from \eqref{defiCylindricalSet} that $\mathcal{Z}(\Phi'_{\beta}) \subseteq \mathcal{B}(\Phi'_{\beta})$. Therefore, $\mathcal{C}(\Phi'_{\beta}) \subseteq \mathcal{B}(\Phi'_{\beta})$. In general this inclusion is strict but if $\Phi$ is separable (for example if it is a countably Hilbertian space) then $\mathcal{C}(\Phi'_{\beta}) = \mathcal{B}(\Phi'_{\beta})$ (see Lemma 4.1 in Mitoma, Okada and Okazaki \cite{MitomaOkadaOkazaki:1977}). 

A function $\mu: \mathcal{Z}(\Phi') \rightarrow [0,\infty]$ is called a \emph{cylindrical measure} on $\Phi'$, if for each finite subset $M \subseteq \Phi'$ the restriction of $\mu$ to $\mathcal{C}(\Phi',M)$ is a measure. A cylindrical measure $\mu$ is called \emph{finite} if $\mu(\Phi')< \infty$ and a \emph{cylindrical probability measure} if $\mu(\Phi')=1$. The complex-valued function $\widehat{\mu}: \Phi \rightarrow \C$ defined by 
$$ \widehat{\mu}(\phi)= \int_{\Phi'} e^{i f[\phi]} \mu(df) = \int_{-\infty}^{\infty} e^{iz} \mu_{\phi}(dz), \quad \forall \, \phi \in \Phi, $$
where for each $\phi \in \Phi$, $\mu_{\phi} \defeq \mu \circ \pi_{\phi}^{-1}$, is called the \emph{characteristic function}\index{characteristic function} of $\mu$. In general, a cylindrical measure on $\Phi'$ does not extend to a Borel measure on $\Phi'_{\beta}$. However, necessary and sufficient conditions for this can be given in terms of the continuity of its characteristic function by means of the Minlos theorem (see \cite{DaleckyFomin}, Theorem 1.3, Chapter III, p.88).

\begin{theo}[Minlos theorem] \label{minlosTheorem}
Let $\Phi$ be a nuclear space. For a function $F:\Phi \rightarrow \C$ to be the characteristic function of a Radon probability measure on $\Phi'_{\beta}$ it is sufficient, and necessary if $\Phi$ is barrelled, that it be positive definite, continuous at zero and satisfies $F(0)=1$. 
\end{theo} 

A \emph{cylindrical random variable}\index{cylindrical random variable} in $\Phi'$ is a linear map $X: \Phi \rightarrow L^{0} \ProbSpace$. If $Z=\mathcal{Z}\left(\phi_{1}, \dots, \phi_{n}; A \right)$ is a cylindrical set, for $\phi_{1}, \dots, \phi_{n} \in \Phi$ and $A \in \mathcal{B}\left(\R^{n}\right)$, let 
\begin{equation*} 
\mu_{X}(Z) \defeq \Prob \left( ( X[\phi_{1}], \dots, X[\phi_{n}]) \in A  \right) = \Prob \circ X^{-1} \circ \pi_{\phi_{1},\dots, \phi_{n}}^{-1}(A). 
\end{equation*}
The map $\mu_{X}$ is called the \emph{cylindrical distribution} of $X$ and it is cylindrical probability measure on $\Phi'$. 

If $X$ is a cylindrical random variable in $\Phi'$, the \emph{characteristic function} of $X$ is defined to be the characteristic function $\widehat{\mu}_{X}: \Phi \rightarrow \C$ of its cylindrical distribution $\mu_{X}$. Therefore, $ \widehat{\mu}_{X}(\phi)= \Exp e^{i X(\phi)} $, $\forall \, \phi \in \Phi$. Also, we say that $X$ is \emph{$n$-integrable} if $ \Exp \left( \abs{X(\phi)}^{n} \right)< \infty$, $\forall \, \phi \in \Phi$. 

Let $X$ be a $\Phi'_{\beta}$-valued random variable, i.e. $X:\Omega \rightarrow \Phi'_{\beta}$ is a $\mathscr{F}/\mathcal{B}(\Phi'_{\beta})$-measurable map. We denote by $\mu_{X}$ the distribution of $X$, i.e. $\mu_{X}(\Gamma)=\Prob \left( X \in  \Gamma \right)$, $\forall \, \Gamma \in \mathcal{B}(\Phi'_{\beta})$, and it is a Borel probability measure on $\Phi'_{\beta}$. For each $\phi \in \Phi$ we denote by $X[\phi]$ the real-valued random variable defined by $X[\phi](\omega) \defeq X(\omega)[\phi]$, for all $\omega \in \Omega$. Then, the mapping $\phi \mapsto X[\phi]$ defines a cylindrical random variable. Therefore, the above concepts of characteristic function and integrability can be analogously defined for $\Phi'_{\beta}$-valued random variables in terms of the cylindrical random variable they determines.   

Very important for our forthcoming arguments is the following class of $\Phi'_{\beta}$-valued random variables introduced by \Ito{} and Nawata in \cite{ItoNawata:1983}. 

\begin{defi} \label{defiRegularRV}
A $\Phi'_{\beta}$-valued random variable $X$ is called \emph{regular} if there exists a weaker countably Hilbertian topology $\theta$ on $\Phi$ such that $\Prob( \omega: X(\omega) \in \Phi'_{\theta})=1$.
\end{defi}

If $X$ is a cylindrical random variable in $\Phi'$, a $\Phi'_{\beta}$-valued random variable $Y$ is a called \emph{version} of $X$ if for every $\phi \in \Phi$, $X(\phi)=Y[\phi]$ $\Prob$-a.e. A sufficient condition for the existence of a regular version is given in the following result due to \Ito{} and Nawata (see \cite{ItoNawata:1983}):  

\begin{theo}[Regularization theorem]\label{regularizationTheorem} Let $X$ be a cylindrical random variable in $\Phi'$ such that $X:\Phi \rightarrow L^{0}\ProbSpace$ is continuous. Then, $X$ has a unique (up to equivalence) $\Phi'_{\beta}$-valued regular version.   
\end{theo}

The following results establish alternative characterizations for regular random variables. 

\begin{theo}\label{theoCharacterizationRegularRV}
Let $X$ be a $\Phi'_{\beta}$-valued random variable. Consider the statements:
\begin{enumerate}
\item $X$ is regular.
\item The map $X: \Phi \rightarrow L^{0} \ProbSpace$, $\phi \mapsto X[\phi]$ is continuous. 
\item The distribution $\mu_{X}$ of $X$ is a Radon probability measure. 
\end{enumerate} 
One has the implications: $(1) \Rightarrow (2)$ and if $\Phi$ is nuclear, then $(2) \Rightarrow (1)$ and $(2) \Rightarrow (3)$. Moreover, if $\Phi$ is barrelled, then $(3) \Rightarrow (1)$.  
\end{theo}
\begin{prf}
$(1) \Rightarrow (2)$: Let $\theta$ be a countably Hilbertian topology on $\Phi$ as in Definition \ref{defiRegularRV}. As the topology $\theta$ is weaker than the nuclear topology on $\Phi$, it is sufficient to show that the map $\phi \mapsto X[\phi]$ is continuous from $\Phi_{\theta}$ into $L^{0} \ProbSpace$. To show this, note that if $\{ \phi_{n} \}_{n \in \N}$ is a sequence converging to $\phi$ in $\Phi_{\theta}$, then the fact that $\Prob( \omega: X(\omega) \in \Phi'_{\theta})=1$ shows that $X[\phi_{n}] \rightarrow X[\phi]$ $\Prob$-a.e. Therefore $\abs{X[\phi_{n}] - X[\phi]} \rightarrow 0$ in probability. Hence, because $\Phi_{\theta}$ is first-countable (this is because it is pseudo-metrizable, see Proposition \ref{propWeakerCountablyHilbertianTopology}) it then follows that the map  $\phi \mapsto X[\phi]$ is continuous from $\Phi_{\theta}$ into $L^{0} \ProbSpace$.

$(2) \Rightarrow (1)$: As the map $\phi \mapsto X[\phi]$ is continuous, because $\Phi$ is nuclear the regularization theorem (Theorem \ref{regularizationTheorem}) shows that $X$ has a regular version. But this clearly implies that $X$ is also regular. 

$(2) \Rightarrow (3)$: The assumption that $\phi \mapsto X[\phi]$ is continuous implies that the characteristic function of $X$ is continuous. If $\Phi$ is nuclear the Minlos theorem (Theorem \ref{minlosTheorem}) shows that the distribution $\mu_{X}$ of $X$ is a Radon probability measure.

$(3) \Rightarrow (1)$: Assume that $\Phi$ is barrelled and that $\mu_{X}$ is a Radon measure. Let $\{ \epsilon_{n} \}_{n \in \N}$ be a decreasing sequence of positive real numbers converging to zero. For every $n \in \N$, because $\mu_{X}$ is a Radon probability measure there exists a compact subset $K_{n}$ of $\Phi'_{\beta}$ such that $\mu_{X}(K_{n}) > 1- \epsilon_{n}$. 

Because $\Phi$ is barrelled and each $K_{n}$ is bounded, there exists an increasing sequence of continuous Hilbertian semi-norms $\{ p_{n} \}_{n \in \N}$ on $\Phi$ such that for every $n \in \N$, $K_{n} \subseteq B_{p_{n}}(1)^{0}$, where $B_{p_{n}}(1)^{0}$ is the polar set of $B_{p_{n}}(1)$ (see \cite{Schaefer}, Result 5.2, Chapter IV, p.141). Then, as $B_{p_{n}}(1)^{0}$ is the unit ball of the Hilbert space $\Phi'_{p_{n}}$, it follows that 
$$\Prob ( X \in \Phi'_{p_{n}})=\mu_{X}(\Phi'_{p_{n}}) \geq \mu_{X}( B_{p_{n}}(1)^{0}) \geq \mu_{X}(K_{n})  >1-\epsilon_{n}. $$
Hence, as $\epsilon_{n} \rightarrow 0$ it follows from the above inequality that $\Prob \left( X \in \bigcup_{n \in \N} \Phi'_{p_{n}} \right)=1$. Then, if $\theta$ is the countably Hilbertian topology generated by the semi-norms $\{ p_{n} \}_{n \in \N}$ it follows from \eqref{dualOfWeakerCountablyHilbertianTopology} that $\Prob( \omega: X(\omega) \in \Phi'_{\theta})=1$ and hence $X$ is a regular random variable. 
\end{prf}

The following is an useful property of regular random variables. 

\begin{prop} \label{propCriteriaVersionRegularRV}
Let $X$, $Y$ be $\Phi'_{\beta}$-valued random regular variables. Then, $X=Y$ $\Prob$-a.e. if and only if for all $\phi \in \Phi$, $X[\phi]=Y[\phi]$ $\Prob$-a.e. 
\end{prop}
\begin{prf}
We only have to prove the sufficiency. Assume $X$ and $Y$ are regular and let $\{ p_{n} \}_{n \in \N}$ be a sequence of Hilbertian semi-norms on $\Phi$ such that Definition \ref{defiRegularRV} is satisfied for both $X$ and $Y$. If $\theta$ is the countably Hilbertian topology on $\Phi$ determined by the semi-norms $\{ p_{n} \}_{n \in \N}$, then $\Prob ( \Omega_{\theta})=1$, where $\Omega_{\theta}=\{ \omega \in \Omega: X(\omega) \in \Phi'_{\theta}, Y(\omega) \in \Phi'_{\theta} \}$.

Because $\Phi_{\theta}$ is separable, there exists a countable subset $\{ \phi_{j}: j \in \N \}$ of $\Phi$ that is dense in $\Phi_{\theta}$. For every $j \in \N$, it follows from our hypothesis that $\Prob ( \Omega_{j})=1$, where $\Omega_{j} = \{ \omega \in \Omega: X(\omega)[\phi_{j}]= Y(\omega)[\phi_{j}] \}$. Let $\Gamma= \Omega_{\theta} \cap \bigcap_{j \in \N} \Omega_{j}$. Then, we have $\Prob ( \Gamma )=1$.     

Fix $\omega \in \Gamma$ and let $\phi \in \Phi$. Then, there exists a sequence $\{ \phi_{j_{k}} \}_{k \in \N} \subseteq   \{ \phi_{j}: j \in \N \}$ that converges to $\phi$ in $\Phi_{\theta}$. Therefore, as $X(\omega), Y(\omega) \in \Phi'_{\theta}$ and because $X(\omega)[\phi_{j_{k}}]= Y(\omega)[\phi_{j_{k}}]$ for all $k \in \N$, it follows that 
$$ X(\omega)[\phi] = \lim_{k \rightarrow \infty}  X(\omega)[\phi_{j_{k}}]= \lim_{k \rightarrow \infty}  Y(\omega)[\phi_{j_{k}}] = Y(\omega)[\phi].$$   
As the above is true for any $\phi \in \Phi$, it follows that $X(\omega)=Y(\omega)$ for every $\omega \in \Gamma$. Therefore, $X=Y$ $\Prob$-a.e.  
\end{prf}

Let $J=[0,\infty)$ or $J=[0,T]$ for some $T>0$. We say that $X=\{ X_{t} \}_{t \in J}$ is a \emph{cylindrical process} in $\Phi'$ if $X_{t}$ is a cylindrical random variable, for each $t \in J$. Clearly, any $\Phi'_{\beta}$-valued stochastic processes $X=\{ X_{t} \}_{t \in J}$ defines a cylindrical process under the prescription: $X[\phi]=\{ X_{t}[\phi] \}_{t \in J}$, for each $\phi \in \Phi$. We will say that it is the \emph{cylindrical process determined} by $X$.

A $\Phi'_{\beta}$-valued processes $Y=\{Y_{t}\}_{t \in J}$ is said to be a $\Phi'_{\beta}$-valued \emph{version} of the cylindrical process $X=\{X_{t}\}_{t \in J}$ on $\Phi'$ if for each $t \in J$, $Y_{t}$ is a $\Phi'_{\beta}$-valued version of $X_{t}$. 

A $\Phi'_{\beta}$-valued stochastic process $X$ is \emph{continuous} (respectively \emph{c\`{a}dl\`{a}g} if for $\Prob$-a.e. $\omega \in \Omega$, the \emph{sample paths} $t \mapsto X_{t}(w) \in \Phi'_{\beta}$ of $X$ are continuous (respectively right-continuous with left limits).

Let $X=\{ X_{t} \}_{t \in J}$ be a $\Phi'_{\beta}$-valued stochastic processes. We say that $X$ is \emph{regular} if for every $t \in J$, $X_{t}$ is a regular random variable. Some useful properties of $\Phi'_{\beta}$- valued regular processes are given below. 

\begin{prop}\label{propCondiIndistingProcess} Let $X=\left\{ X_{t} \right\}_{t \in J}$ and $Y=\left\{ Y_{t} \right\}_{t \in J}$ be $\Phi'_{\beta}$- valued regular stochastic processes such that for each $\phi \in \Phi$, $X[\phi]=\left\{ X_{t}[\phi] \right\}_{t \in J}$ is a version of $Y=\left\{ Y_{t}[\phi] \right\}_{t \in J}$. Then $X$ is a version of $Y$. Furthermore, if $X$ and $Y$ are right-continuous then they are indistinguishable processes. 
\end{prop}
\begin{prf}
Fix $t \in J$. Then, as for each $\phi \in \Phi$, $X_{t}[\phi]=Y_{t}[\phi]$ $\Prob$-a.e., then Proposition \ref{propCriteriaVersionRegularRV} shows that $X_{t}=Y_{t}$ $\Prob$-a.e. Therefore, $X$ is a version of $Y$. Now, assume that both $X$ and $Y$ are right-continuous. Let $\Omega_{X}$ and $\Omega_{Y}$ denote respectively the sets of all $\omega \in \Omega$ such that the maps $t \mapsto X_{t}(\omega)$ and $t \mapsto Y_{t}(\omega)$ are right-continuous. Let $\Gamma_{X,Y}=\{ \omega \in \Omega: X_{t}(\omega)=Y_{t}(\omega), \forall t \in \Q_{+} \}$, where $\Q_{+}=\Q \cap J$. Then, $\Prob ( \Omega_{X} \cap \Omega_{Y} \cap \Gamma_{X,Y} )=1$ and by the right-continuity of $X$ and $Y$ and the denseness of $\Q_{+}$ in $J$, it follows by a standard argument that $X_{t}(\omega)=Y_{t}(\omega)$ for all $t \in J$, for each $\omega \in \Omega_{X} \cap \Omega_{Y} \cap \Gamma_{X,Y}$. Thus, $X$ and $Y$ are indistinguishable. 
\end{prf}

\section{The Regularization Theorem for Cylindrical Stochastic Processes}\label{sectionRegulTheorem}

\begin{assu}
From now on, $\Phi$ will always denote a nuclear space over $\R$. 
\end{assu}

The main result of this paper is the following theorem that establish conditions for the existence of a $\Phi'_{\beta}$-valued, regular, continuous (or c\`{a}dl\`{a}g) version for a cylindrical process in $\Phi'$. 

\begin{theo}[Regularization Theorem]\label{theoRegularizationTheoremCadlagContinuousVersion}
Let $X=\{X_{t} \}_{t \geq 0}$ be a cylindrical process in $\Phi'$ satisfying:
\begin{enumerate}
\item For each $\phi \in \Phi$, the real-valued process $X(\phi)=\{ X_{t}(\phi) \}_{t \geq 0}$ has a continuous (respectively c\`{a}dl\`{a}g) version.
\item For every $T > 0$, the family $\{ X_{t}: t \in [0,T] \}$ of linear maps from $\Phi$ into $L^{0} \ProbSpace$ is equicontinuous.  
\end{enumerate}
Then, there exists a countably Hilbertian topology $\vartheta_{X}$ on $\Phi$ 
and a $(\widetilde{\Phi_{\vartheta_{X}}})'_{\beta}$-valued continuous (respectively c\`{a}dl\`{a}g) process $Y= \{ Y_{t} \}_{t \geq 0}$, such that for every $\phi \in \Phi$, $Y[\phi]= \{ Y_{t}[\phi] \}_{t \geq 0}$ is a version of $X(\phi)= \{ X_{t}(\phi) \}_{t \geq 0}$. Moreover, $Y$ is a $\Phi'_{\beta}$-valued, regular, continuous (respectively c\`{a}dl\`{a}g) version of $X$ that is unique up to indistinguishable versions. 
\end{theo}

We proceed to prove Theorem \ref{theoRegularizationTheoremCadlagContinuousVersion}. Let $X=\{X_{t} \}_{t \in [0,T]}$ be a cylindrical process in $\Phi'$ satisfying conditions \emph{(1)} and \emph{(2)} of Theorem \ref{theoRegularizationTheoremCadlagContinuousVersion}.
Without loss of generality we assume that each $X(\phi)=\{ X_{t}(\phi) \}_{t \geq 0}$ has a continuous version. The c\`{a}dl\`{a}g version case follows from the same arguments.  

The next proposition constitutes the main step in the proof of Theorem \ref{theoRegularizationTheoremCadlagContinuousVersion}. 

\begin{prop}\label{propRegulTheoremCadlagContinuousVersionInBoundedInterval}
For every $T>0$ there exists a countably Hilbertian topology $\vartheta_{T}$ on $\Phi$ 
and a $(\widetilde{\Phi_{\vartheta_{T}}})'_{\beta}$-valued continuous (respectively c\`{a}dl\`{a}g) process $Y^{(T)}= \{ Y^{(T)}_{t} \}_{t \in [0,T]}$, such that for every $\phi \in \Phi$, $\{ Y^{(T)}_{t}[\phi] \}_{t \in [0,T]}$ is a version of $ \{ X_{t}(\phi) \}_{t \in [0,T]}$.  
\end{prop}

For the benefit of the reader we split the proof of Proposition \ref{propRegulTheoremCadlagContinuousVersionInBoundedInterval} in several steps. Fix $T>0$. For each $\phi \in \Phi$, let $\widehat{X}(\phi)=\{ \widehat{X}_{t}(\phi) \}_{t \in [0,T]}$ be a continuous version of $X(\phi)=\{ X_{t}(\phi) \}_{t \in [0,T]}$. It is clear from the corresponding properties of $X$ that $\widehat{X}$ determines a cylindrical process $\widehat{X}=\{ \widehat{X}_{t}\}_{t \in [0,T] }$ in $\Phi'$. Moreover, we have the following result.   

\begin{lemm}\label{lemmContiXtOnCHT} There exists a weaker countably Hilbertian topology $\theta$ on $\Phi$ such that the family of linear maps $\{ \widehat{X}_{t}: t \in [0,T] \}$ from $\Phi$ into $L^{0} \ProbSpace$ is $\theta$-equicontinuous.
\end{lemm}
\begin{prf}
Let $\{ \epsilon_{n} \}_{n \in \N}$ be a decreasing sequence of positive numbers converging to zero. From the equicontinuity of the family of linear maps $\{ X_{t}: t \in [0,T] \}$ and from the fact that $\widehat{X}(\phi)$ is a version of $X(\phi)$ for each $\phi \in \Phi$, it follows that for every $n \in \N$ there exists a continuous Hilbertian semi-norm $p_{n}$ on $\Phi$ such that 
\begin{equation} \label{contiXtOnCHT}
\Prob \left( \omega: \abs{\widehat{X}_{t}(\phi)(\omega)} > \epsilon_{n} \right) \leq \epsilon_{n}, \quad \forall \, \phi \in B_{p_{n}}(1), \, \forall \, t \in [0,T].  
\end{equation}
%
Then, if $\theta$ is the countably Hilbertian topology on $\Phi$ generated by the semi-norms $\{ p_{n} \}_{n \in \N}$, it follows from \eqref{contiXtOnCHT} and the fact that $\lim_{n \rightarrow \infty} \epsilon_{n}=0$ that the  family $\{ \widehat{X}_{t}: t \in [0,T] \}$ is $\theta$-equicontinuous.
\end{prf}
 
For the next result we need the following terminology. Let $C_{T}(\R)$ and $D_{T}(\R)$ denote respectively the space of continuous and c\`{a}dl\`{a}g real-valued processes defined in $[0,T]$, both are considered equipped with the topology of uniform convergence in probability on $[0,T]$.

\begin{lemm} \label{lemmContiXAsMapToSpaceContiProcesses}
The linear map $\widehat{X}$ from $\Phi$ into $C_{T}(\R)$ given by $\phi \mapsto \widehat{X}(\phi)=\{ \widehat{X}_{t}(\phi)\}_{t \in [0,T]}$ is continuous.
\end{lemm}
\begin{prf}
Let $\theta$ be a countably Hilbertian topology on $\Phi$ as in Lemma \ref{lemmContiXtOnCHT}. Then, for every $t \in [0,T]$, 
the map $\widehat{X}_{t}:\Phi_{\theta} \rightarrow L^{0} \ProbSpace$ is linear and continuous. Therefore, for all $t \in [0,T]$ there exists a continuous and linear map $\widetilde{X}_{t}:\widetilde{\Phi_{\theta}} \rightarrow L^{0} \ProbSpace$ satisfying $\widehat{X}_{t}= \widetilde{X}_{t} \circ J$, where $J$ is the canonical embedding from $\Phi_{\theta}$ into its completion $\widetilde{\Phi_{\theta}}$ (see \cite{Jarchow}, Theorem 2, Section 3.4, p.61-2).

We are going to show that the linear map $\widetilde{X}$ from $\widetilde{\Phi_{\theta}}$ into $C_{T}(\R)$ given by $\phi \mapsto \widetilde{X}(\phi)=\{ \widetilde{X}_{t}(\phi)\}_{t \in [0,T]}$ is closed. Let $\{ \phi_{n} \}_{n \in \N}$ be a sequence converging to $\phi$ in $\widetilde{\Phi_{\theta}}$ and assume that there exists  $Y \in C_{T}(\R)$ such that $\sup_{t \in [0,T]} \abs{\widetilde{X}_{t}(\phi_{n})-Y_{t}}$ converges in probability to $0$ as $n \rightarrow \infty$. We have to prove that $\widetilde{X}(\phi)=Y$ in $C_{T}(\R)$. 

First, for every $t \in [0,T]$  the continuity of the map $\widetilde{X}_{t}:\widetilde{\Phi_{\theta}} \rightarrow L^{0} \ProbSpace$ and the fact that $\{ \phi_{n} \}_{n \in \N}$ converges to $\phi$ in $\widetilde{\Phi_{\theta}}$ implies that the sequence of random variables $\{ \widetilde{X}_{t}(\phi_{n}) \}_{n \in \N}$ converges in probability to $\widetilde{X}_{t}(\phi)$. 

On the other hand, as $\sup_{t \in [0,T]} \abs{\widetilde{X}_{t}(\phi_{n})-Y}$ converges in probability to $0$ as $n \rightarrow \infty$, then for every $t \in [0,T]$ the sequence of random variables $\{ \widetilde{X}_{t}(\phi_{n}) \}_{n \in \N}$ converges in probability to $Y_{t}$. Therefore, by uniqueness of limits in $L^{0}\ProbSpace$ it follows that $\widetilde{X}_{t}(\phi)=Y_{t}$ $\Prob$-a.e. for every $t \in [0,T]$, and hence $\widetilde{X}(\phi)=Y$ in $C_{T}(\R)$. 

Then, the linear map $\widetilde{X}:\widetilde{\Phi_{\theta}} \rightarrow C_{T}(\R)$ is sequentially closed and therefore closed because $\widetilde{\Phi_{\theta}}$ is pseudo-metrizable (Proposition \ref{propWeakerCountablyHilbertianTopology}) and hence first-countable. Now, because $\widetilde{\Phi_{\theta}}$ is a Baire space (Proposition \ref{propWeakerCountablyHilbertianTopology}) and $C_{T}(\R)$ is a complete, metrizable, topological vector space, the closed graph theorem (see \cite{NariciBeckenstein}, Theorem 14.3.4, p.465) shows that $\widetilde{X}$ is continuous. 

Now, because for every $t \in [0,T]$, $\widehat{X}_{t}= \widetilde{X}_{t} \circ J$ then we have that the map $\widehat{X}$ from 
$\Phi_{\theta}$ into $C_{T}(\R)$ given by $\phi \mapsto \widehat{X}(\phi)=\{ \widehat{X}_{t}(\phi)\}_{t \in [0,T]}$ satisfies 
$\widehat{X}=\widetilde{X} \circ J$. Therefore, as both $\widetilde{X}$ and $J$ are continuous, it follows that $\widehat{X}$ is 
also continuous. Moreover, as the topology $\theta$ is weaker than the nuclear topology on $\Phi$, it follows that $\widehat{X}$ is continuous as a map from $\Phi$ into $C_{T}(\R)$.  
\end{prf}

\begin{rema}
In the case that for each $\phi \in \Phi$, $X(\phi)=\{ X_{t}(\phi) \}_{t \in [0,T]}$ has a c\`{a}dl\`{a}g version $\widehat{X}(\phi)=\{ \widehat{X}_{t}(\phi) \}_{t \in [0,T]}$, then in Lemma \ref{lemmContiXAsMapToSpaceContiProcesses} we have that the linear mapping $\widehat{X}$ from $\Phi$ into $D_{T}(\R)$ given by $\phi \mapsto \widehat{X}(\phi)$ is continuous.
\end{rema}

\begin{lemm}\label{lemmPnContiCharacFunctSupremum} Let $D$ be a countable dense subset of $[0,T]$. For every $\epsilon >0$ there exists a continuous Hilbertian semi-norm $p$ on $\Phi$ such that 
\begin{equation} \label{pnContinuityCharactFuncSupremum}
\Exp \left( \sup_{t \in D} \abs{1-e^{i \widehat{X}_{t}(\phi)}} \right) \leq \epsilon + 2p(\phi)^{2}, \quad \forall \, \phi \in \Phi. 
\end{equation}
\end{lemm}
\begin{prf}
We proceed in a way similar to Lemma 1 of Mitoma \cite{Mitoma:1983}. Let $\epsilon >0$. From the continuity of the exponential function there exists $\delta >0$ such that $\abs{1-e^{ir}}\leq \frac{\epsilon}{2}$ if $\abs{r} \leq \delta$. Now, from the continuity of the map $\phi \mapsto \widehat{X}(\phi)$ from $\Phi$ into $C_{T}(\R)$ (Lemma \ref{lemmContiXAsMapToSpaceContiProcesses}), there exists a continuous Hilbertian semi-norm $p$ on $\Phi$ such that 
\begin{equation} \label{continuityOrigenCHTRegularVersionOperator}
\Prob \left( \omega: \sup_{t \in D} \abs{\widehat{X}_{t}(\phi)(\omega)} > \delta \right) \leq \frac{\epsilon}{4}, \quad \forall \, \phi \in B_{p}(1).  
\end{equation}
Let $\Upsilon=\{ \omega \in \Omega: \sup_{t \in D} \abs{\widehat{X}_{t}(\phi)(\omega)} < \delta\}$. Then, if $\phi \in B_{p}(1)$ it follows from \eqref{continuityOrigenCHTRegularVersionOperator} that
$$ \Exp \left( \sup_{t \in D} \abs{1-e^{i\widehat{X}_{t}(\phi)}} \right) 
\leq \int_{\Upsilon} \sup_{t \in D} \abs{1-e^{i\widehat{X}_{t}(\phi)(\omega)}} \Prob(d \omega) + 2 \, \Prob (\Upsilon^{c}) \leq \epsilon. 
$$
On the other hand, because $\sup_{t \in D} \abs{1-e^{i\widehat{X}_{t}(\phi)}}\leq 2$ for any $\phi \in \Phi$, then if $\phi \in B_{p}(1)^{c}$, we have 
$$ \Exp \left( \sup_{t \in D} \abs{1-e^{i \widehat{X}_{t}(\phi)}} \right) \leq  2p(\phi)^{2}.$$
Therefore, from the above inequalities we obtain \eqref{pnContinuityCharactFuncSupremum}. 
\end{prf}

\begin{lemm}\label{lemmaRegulaTheo} There exists an increasing sequence of continuous Hilbertian semi-norms $\{q_{n}\}_{n \in \N}$ on $\Phi$, a subset $\Omega_{Y}$ of $\Omega$ such that $\Prob \left( \Omega_{Y} \right) =1$ and a sequence of stochastic processes $Y^{(n)}=\{ Y_{t}^{(n)} \}_{t \in [0,T]}$, $n \in \N$, satisfying: 
\begin{enumerate}
\item For each $n \in \N$, $Y^{(n)}$ is a $\Phi'_{q_{n}}$-valued process such that for every $\phi \in \Phi_{q_{n}}$, $Y^{(n)}[\phi]=\{ Y^{(n)}_{t}[\phi] \}_{t \in [0,T]}$ is a continuous real-valued process. 
\item For each $\omega \in \Omega_{Y}$, there exists $N(\omega)$ such that
\begin{enumerate}
\item For all $n \geq N(\omega)$, $\sup_{t \in [0,T]} q_{n}'(Y^{(n)}_{t}(\omega))<\infty$, and 
\item For all $m \geq n \geq N(\omega)$, $Y^{(m)}_{t}(\omega)=i'_{q_{n},q_{m}} Y^{(n)}_{t}(\omega)$ for all $t \in [0,T]$.  
\end{enumerate}
\item For every $\phi \in \Phi$, there exists $\Delta_{\phi} \subseteq \Omega$ with $\Prob(\Delta_{\phi})=1$ such that for every $\omega \in \Omega_{Y} \cap \Delta_{\phi}$ there exists $N(\omega)$ such that for each $n \geq N(\omega)$, $Y^{(n)}_{t}(\omega)[i_{q_{n}} \phi]=\widehat{X}_{t}(\phi)(\omega)$ for all $t \in [0,T]$.    
\end{enumerate}  
\end{lemm}
\begin{prf} We follow similar arguments to those used by \Ito{} and Nawata in \cite{ItoNawata:1983}. Let $D$ be a countable dense subset of $[0,T]$. Let $\{ \epsilon_{n} \}_{n \in \N}$ be a sequence of positive numbers such that $\sum_{n \in \N} \epsilon_{n} < \infty$. From Lemma \ref{lemmPnContiCharacFunctSupremum} there exist and increasing sequence of continuous Hilbertian semi-norms $\{ p_{n} \}_{n \in \N}$ on $\Phi$ such that for each $n \in \N$, $\epsilon_{n}$ and $p_{n}$ satisfy  \eqref{pnContinuityCharactFuncSupremum}.

Now, as $\Phi$ is nuclear, there exists an increasing sequence of continuous Hilbertian semi-norms $\{ q_{n} \}_{n \in \N}$ on $\Phi$ such that for each $n \in \N$, $p_{n} \leq q_{n}$ and the inclusion $i_{p_{n},q_{n}}$ is Hilbert-Schmidt. Let $\alpha$ be the countably Hilbertian topology on $\Phi$ generated by the semi-norms $\{ q_{n} \}_{n \in \N}$. The space $\Phi_{\alpha}$ is separable (Proposition \ref{propWeakerCountablyHilbertianTopology}). Let $B=\{ \xi_{k}: k \in \N \}$  be a countable dense subset of $\Phi_{\alpha}$. For every $n \in \N$, from an application of the Schmidt orthogonalization procedure to $B$, we can find a complete orthonormal system $\{ \phi_{j}^{q_{n}} \}_{j \in \N} \subseteq \Phi$ of $\Phi_{q_{n}}$, such that 
\begin{equation} \label{decompDenseSetInTermsOrtoBasis}
\xi_{k}= \sum_{j=1}^{k} a_{j,k,n} \, \phi_{j}^{q_{n}} + \varphi_{k,n}, \quad \forall \, k \in \N, 
\end{equation}
with $a_{j,k,n} \in \R$ and $\varphi_{k,n} \in \mbox{Ker}(q_{n})$, for each $j,k \in \N$. 

Let $n \in \N$. From the inequality: $1-e^{-r/2} \geq 1-e^{-1/2} = \frac{\sqrt{e}-1}{\sqrt{e}}$ for $r >1$, for any $C>0$ we have 
\begin{flalign}
& \Prob \left( \sup_{t \in D}\sum_{j=1}^{\infty} \abs{\widehat{X}_{t}(\phi_{j}^{q_{n}})}^{2} > C^{2} \right) \nonumber \\
&  \leq   \frac{\sqrt{e}}{\sqrt{e}-1} \Exp \left( 1-\exp \left\{ - \frac{1}{2 C^{2}} \sup_{t \in D} \sum_{j=1}^{\infty} \abs{\widehat{X}_{t}(\phi_{j}^{q_{n}})}^{2} \right\}  \right) \nonumber \\
& \hfill = \lim_{m \rightarrow \infty} \frac{\sqrt{e}}{\sqrt{e}-1} \Exp  \sup_{t \in D} \left(  1-\exp \left\{ - \frac{1}{2 C^{2}} \sum_{j=1}^{m} \abs{\widehat{X}_{t}(\phi_{j}^{q_{n}})}^{2} \right\}  \right) \label{inequa1Step1RegulTheo} 
\end{flalign}
Now, setting $\phi=\sum_{j =1}^{m} z_{j} \phi_{j}^{q_{n}}$ for $z_{1}, \dots, z_{m} \in \R$, in \eqref{pnContinuityCharactFuncSupremum} for $p_{n}$ we have
\begin{equation} \label{contCharactFunctSumStep1RegulTheo}
\Exp \left( \sup_{t \in D} \abs{1-\exp\left\{ i\sum_{j=1}^{m} z_{j} \widehat{X}_{t}(\phi_{j}^{q_{n}}) \right\} } \right) \leq \epsilon_{n} + 2 \sum_{j=1}^{m} z_{j}^{2} p_{n}(\phi_{j}^{q_{n}})^{2}.  
\end{equation}
Integrating both sides of \eqref{contCharactFunctSumStep1RegulTheo} with respect to $\prod_{j=1}^{m} N_{C}(d z_{j})$, where $N_{C}$ is the centered Gaussian measure on $\R$ with variance $1/C^{2}$, we have 
\begin{equation} \label{inequa2Step1RegulTheo}
\int_{\R^{m}} \Exp \left( \sup_{t \in D} \abs{1-\exp\left\{ i\sum_{j=1}^{m} z_{j} \widehat{X}_{t}(\phi_{j}^{q_{n}}) \right\} } \right) \prod_{j=1}^{m} N_{C}(d z_{j}) \leq \epsilon_{n} + \frac{2}{C^{2}} \sum_{j=1}^{m} p_{n}(\phi_{j}^{q_{n}})^{2}.  
\end{equation}
On the other hand, as $\prod_{j=1}^{m} N_{C}(d z_{j})$ is a Gaussian measure on $\R^{m}$, for each $t \in [0,T]$ and $\omega \in \Omega$ we have
\begin{equation} \label{charactFunctGaussianStep1RegulTheo}
 \exp \left\{ - \frac{1}{2 C^{2}} \sum_{j=1}^{m} \abs{\widehat{X}_{t}(\phi_{j}^{q_{n}})(\omega)}^{2} \right\}
= \int_{\R^{m}} \exp \left\{ i\sum_{j=1}^{m} z_{j} \widehat{X}_{t}(\phi_{j}^{q_{n}})(\omega) \right\} \prod_{j=1}^{m} N_{C}(d z_{j}),
\end{equation}
and therefore from \eqref{charactFunctGaussianStep1RegulTheo} and the Fubini theorem it follows that 
\begin{multline} 
\Exp  \sup_{t \in D} \left(  1-\exp \left\{ - \frac{1}{2 C^{2}} \sum_{j=1}^{m} \abs{\widehat{X}_{t}(\phi_{j}^{q_{n}})}^{2} \right\}  \right) \\
\leq \int_{\R^{m}} \Exp \left( \sup_{t \in D} \abs{1-\exp\left\{ i\sum_{j=1}^{m} z_{j} \widehat{X}_{t}(\phi_{j}^{q_{n}}) \right\} } \right) \prod_{j=1}^{m} N_{C}(d z_{j}). \label{inequa3Step1RegulTheo}
\end{multline}
Then, from \eqref{inequa1Step1RegulTheo}, \eqref{inequa2Step1RegulTheo} and \eqref{inequa3Step1RegulTheo}, it follows that 
\begin{eqnarray*}
 \Prob \left( \sup_{t \in D}\sum_{j=1}^{\infty} \abs{\widehat{X}_{t}(\phi_{j}^{q_{n}})}^{2} > C^{2} \right) 
& \leq &   \lim_{m \rightarrow \infty} \frac{\sqrt{e}}{\sqrt{e}-1} \left( \epsilon_{n} + \frac{2}{C^{2}} \sum_{j=1}^{m} p_{n}(\phi_{j}^{q_{n}})^{2} \right) \\
& = & \frac{\sqrt{e}}{\sqrt{e}-1} \left( \epsilon_{n} + \frac{2}{C^{2}} \norm{i_{p_{n},q_{n}}}^{2}_{\mathcal{L}_{2}(\Phi_{q_{n}},\Phi_{p_{n}} )} \right), 
\end{eqnarray*}
where $\norm{i_{p_{n},q_{n}}}_{\mathcal{L}_{2}(\Phi_{q_{n}},\Phi_{p_{n}} )} < \infty$ as $i_{p_{n},q_{n}}$ is Hilbert-Schmidt. Letting $C \rightarrow \infty$, we get 
\begin{equation} \label{firstPartInequalityProofContCadlagVersionLCS}
\Prob \left( \sup_{t \in D} \sum_{j=1}^{\infty} \abs{\widehat{X}_{t}(\phi_{j}^{q_{n}})}^{2} <  \infty \right)  \geq 1- \frac{\sqrt{e}}{\sqrt{e}-1} \epsilon_{n}. 
\end{equation} 
Following the same arguments as above but now replacing $\phi_{j}^{q_{n}}$ for $\varphi_{j,n}$ and using the fact that $\varphi_{j,n} \in \mbox{Ker}(p_{n})$ for each $j \in \N$ (recall $p_{n} \leq q_{n}$), we have that  
\begin{equation} \label{secondPartInequalityProofContCadlagVersionLCS}
\Prob \left( \sup_{t \in D} \sum_{j=1}^{\infty} \abs{\widehat{X}_{t}(\varphi_{j,n})}^{2} >0 \right)  \leq \frac{\sqrt{e}}{\sqrt{e}-1} \epsilon_{n}. 
\end{equation}
Then, if we define $\Omega_{n} \subseteq \Omega$ by 
\begin{equation} \label{defiSetOmegaNRegulaTheo}
\Omega_{n}=\left\{ \omega: \sup_{t \in D} \sum_{j=1}^{\infty} \abs{\widehat{X}_{t}(\phi_{j}^{q_{n}})(\omega)}^{2} <  \infty \mbox{ and } \widehat{X}_{t}(\varphi_{j,n})(\omega) =0, \, \forall \, t \in D,  j \in \N \right\}, 
\end{equation}
it follows from \eqref{firstPartInequalityProofContCadlagVersionLCS} and \eqref{secondPartInequalityProofContCadlagVersionLCS} that 
\begin{equation} \label{probSetOmegaNRegulaTheo} 
\Prob \left( \Omega_{n} \right)  \geq 1- 2 \frac{\sqrt{e}}{\sqrt{e}-1} \epsilon_{n}. 
\end{equation}

The next point in our agenda is to define the set $\Omega_{Y}$ of $\Prob$-measure 1 and the stochastic processes $Y^{(n)}=\{ Y_{t}^{(n)} \}_{t \in [0,T]}$, $n \in \N$, that satisfy the properties \emph{(1)}-\emph{(3)} in the statement of the Lemma. But before, we set some additional notation. 

For every $n \in \N$, let $\Gamma_{n} \subseteq \Omega$ given by 
\begin{equation} \label{defiSetGammaNRegulTheo}
\Gamma_{n} =  \left\{ \omega: \forall \, j \in \N, \, t \mapsto \widehat{X}_{t}(\phi_{j}^{q_{n}})(\omega) \mbox{ is continuous} \right\}. 
\end{equation} 
Then, for each $n \in \N$ we have $\Prob (\Gamma_{n})=1$. Also, for each $n \in \N$ define $A_{n} \subseteq \Omega$ by
\begin{equation} \label{defiSetANRegulTheo}
A_{n} = \left\{ \omega: \widehat{X}_{t}(\xi_{k})(\omega)= \sum_{j=1}^{k} a_{j,k,n} \widehat{X}_{t}(\phi_{j}^{q_{n}})(\omega) + \widehat{X}_{t}(\varphi_{k,n})(\omega), \, \forall k \in \N, \, t \in D  \right\}.
\end{equation}
For every $n \in \N$, it follows from \eqref{decompDenseSetInTermsOrtoBasis} and the linearity of each $\widehat{X}_{t}$ that $\Prob(A_{n})=1$. Now, for $n \in \N$ define 
\begin{equation} \label{defiSetLambdaNRegulTheo}
\Lambda_{n} = \Omega_{n} \cap \Gamma_{n} \cap A_{n}. 
\end{equation}
Then, it follows from \eqref{probSetOmegaNRegulaTheo} and the fact that $\Prob (\Gamma_{n})=1$ and $\Prob(A_{n})=1$ that 
\begin{equation} \label{probSetLambdaNRegulTheo}
\Prob \left( \Lambda_{n} \right) \geq 1- 2 \frac{\sqrt{e}}{\sqrt{e}-1} \epsilon_{n}. 
\end{equation}
We are ready to define the stochastic processes $Y^{(n)}$, $n \in \N$. For each $n \in \N$, let $\{ f_{j}^{q_{n}} \}_{j \in \N}$ be a complete orthonormal system in $\Phi'_{q_{n}}$ dual to $\{\phi_{j}^{q_{n}} \}_{j \in \N}$, i.e. $f_{j}^{q_{n}}[\phi_{i}^{q_{n}}]=\delta_{i,j}$ where $\delta_{i,j}=1$ if $i=j$ and $\delta_{i,j}=0$ if $i \neq j$. For each $t \in [0,T]$, we define 
\begin{equation} \label{defiVersionYnRegularTheorem}
Y^{(n)}_{t}(\omega) \defeq 
\begin{cases}
\sum_{j=1}^{\infty} \widehat{X}_{t}(\phi_{j}^{q_{n}})(\omega) f_{j}^{q_{n}}, & \mbox{for } \omega \in \Lambda_{n}, \\
0, & \mbox{elsewhere.}
\end{cases}
\end{equation}
Note that $Y^{(n)}=\{ Y^{(n)}_{t}\}_{t \in [0,T]}$ is a well-defined $\Phi'_{q_{n}}$-valued stochastic process. This is because if $\omega \in \Lambda_{n}$, then the infinite sum in \eqref{defiVersionYnRegularTheorem} is convergent, as from Parseval's identity, \eqref{defiSetOmegaNRegulaTheo}, \eqref{defiSetLambdaNRegulTheo} and \eqref{defiVersionYnRegularTheorem}, we have:
\begin{equation} \label{versionYnRegularTheoremWellDefined}
\sup_{t \in [0,T]} q_{n}'(Y^{(n)}_{t}(\omega))^{2} = 
\sup_{t \in D} \sum_{j =1}^{\infty} \abs{\widehat{X}_{t}(\phi_{j}^{q_{n}})(\omega)}^{2} < \infty. 
\end{equation}
Moreover, it follows from \eqref{defiSetGammaNRegulTheo}, \eqref{defiSetLambdaNRegulTheo} and  \eqref{defiVersionYnRegularTheorem} that for every $\phi \in \Phi_{q_{n}}$, $Y^{(n)}[\phi]$ is a continuous real-valued process. Therefore, $Y^{(n)}$ satisfies the property \emph{(1)} in the statement of the Lemma. 

Also, from \eqref{defiVersionYnRegularTheorem} it follows that for each $\omega \in \Lambda_{n}$, $Y^{(n)}_{t}(\omega)[ i_{q_{n}}\phi_{j}^{q_{n}}]=\widehat{X}_{t}(\phi_{j}^{q_{n}})(\omega)$, for all $j \in \N$ and $t \in [0,T]$. Similarly, from the fact that $q_{n}(\varphi_{j,n})=0$ for all $j \in \N$, we have $\abs{f_{j}^{q_{n}}[i_{q_{n}} \varphi_{j,n}]} \leq q'_{n}(f_{j}^{q_{n}})q_{n}(i_{q_{n}} \varphi_{j,n})=0$ for all $j \in \N$, then \eqref{defiSetOmegaNRegulaTheo},  \eqref{defiSetLambdaNRegulTheo} and \eqref{defiVersionYnRegularTheorem} implies that for each $\omega \in \Lambda_{n}$, $Y^{(n)}_{t}(\omega)[i_{q_{n}} \varphi_{j,n}]=\widehat{X}_{t}(\varphi_{j,n})(\omega)=0$ for all $t \in D$, $j \in \N$. So, by \eqref{defiSetGammaNRegulTheo}, \eqref{defiSetANRegulTheo} and \eqref{defiSetLambdaNRegulTheo} we have 
\begin{equation} \label{equalityInDenseSetXAndVersionYn}
\forall \, \omega \in \Lambda_{n}, \quad Y^{(n)}_{t}(\omega)[i_{q_{n}} \xi_{k}]=\widehat{X}_{t}(\xi_{k})(\omega), \quad \forall \, k \in \N, \, t \in [0,T]. 
\end{equation}  
Now we are going to show that \eqref{equalityInDenseSetXAndVersionYn} implies that for every $\phi \in \Phi$, $Y^{(n)}_{t}[i_{q_{n}} \phi]=\widehat{X}_{t}(\phi)$ $\Prob$-a.e. on $\Lambda_{n}$, for all $t \in [0,T]$. 

Let $\phi \in \Phi$. Since $B=\{ \xi_{k}: k \in \N \}$ is dense in $\Phi_{\alpha}$, there exists a sequence $\{ \xi_{k_{j}} \}_{j \in \N} \subseteq B$ that $\alpha$-converges to $\phi$. As $\alpha$ is the countably Hilbertian topology generated by the semi-norms $\{ q_{n} \}_{n \in \N}$, then $\{ i_{q_{n}} \xi_{k_{j}} \}_{j \in \N}$ converges to $i_{q_{n}}\phi$ in $\Phi_{q_{n}}$ as $j \rightarrow \infty$. Therefore, $Y^{(n)}_{t}(\omega)[i_{q_{n}}\xi_{k_{j}}] \rightarrow Y^{(n)}_{t}(\omega)[i_{q_{n}} \phi]$ as $j \rightarrow \infty$, for all $t \in [0,T]$ and $\omega \in \Omega$.

On the other hand, observe that \eqref{continuityOrigenCHTRegularVersionOperator} implies that $\widehat{X}:\Phi \rightarrow C_{T}(\R)$, $\phi \mapsto \widehat{X}(\phi)$, is continuous with respect to the countably Hilbertian topology on $\Phi$ generated by the semi-norms $\{ p_{n} \}_{n \in \N}$ and since this topology is weaker than $\alpha$ (because $p_{n} \leq q_{n}$ for all $n \in \N$), then it is also $\alpha$-continuous. Therefore, there exists $\Delta_{\phi} \subseteq \Omega$ with $\Prob (\Delta_{\phi})=1$ and a subsequence $\{ \xi_{k_{j,\nu}} \}_{\nu \in \N}$ of $\{ \xi_{k_{j}} \}_{j \in \N}$ such that for all $\omega \in \Delta_{\phi}$, $\widehat{X}_{t}(\xi_{k_{j,\nu}})(\omega) \rightarrow \widehat{X}_{t}(\phi)(\omega)$, as $\nu \rightarrow \infty$, for all $t \in [0,T]$.

Then, \eqref{equalityInDenseSetXAndVersionYn} and the uniqueness of limits implies that 
\begin{equation} \label{equalityPaeXAndVersionYnOnPhi}
\forall \, \omega \in \Lambda_{n} \cap \Delta_{\phi}, \quad Y^{(n)}_{t}(\omega)[i_{q_{n}} \phi]=\widehat{X}_{t}(\phi)(\omega), \quad \forall \, t \in [0,T]. 
\end{equation}  
 
Our final step is to define the set $\Omega_{Y}$ and to verify that it and the processes $Y^{(n)}$, $n \in \N$, defined in \eqref{defiVersionYnRegularTheorem} satisfy the conditions \emph{(2)} and \emph{(3)} of the statement of the Lemma. First, it follows from \eqref{probSetLambdaNRegulTheo}, our assumption that $\sum_{n \in \N} \epsilon_{n} < \infty$, and the Borel-Cantelli lemma that  
\begin{equation} \label{defiSetOmegaY}
\Prob \left( \Omega_{Y} \right)=1, \quad \mbox{where}  \quad \Omega_{Y} \defeq \bigcup_{N \in \N} \bigcap_{n \geq N} \Lambda_{n}. 
\end{equation}
Let $\omega \in \Omega_{Y}$. Then, it follows from \eqref{versionYnRegularTheoremWellDefined} and \eqref{defiSetOmegaY} that there exists some $N(\omega)$ such that for all $n \geq N(\omega)$, $\sup_{t \in D} q_{n}'(Y^{(n)}_{t}(\omega))< \infty$. Thus, the property \emph{(2)(a)} in the statement of the Lemma is satisfied. Moreover, from \eqref{equalityInDenseSetXAndVersionYn} and \eqref{defiSetOmegaY} there exists $N(\omega)$ such that for all $m \geq n \geq N(\omega)$, for every $k \in \N$ and $t \in [0,T]$ we have $Y^{(m)}_{t}(\omega)[i_{q_{m}} \xi_{k}]=Y^{(n)}_{t}(\omega)[i_{q_{n}} \xi_{k}]$. But as $B=\{ \xi_{j} \}_{j \in \N}$ is dense $\Phi_{\alpha}$, and therefore in $\Phi_{q_{m}}$ and in $\Phi_{q_{n}}$, then it follows that for all $t \in [0,T]$, $Y^{(m)}_{t}(\omega)[i_{q_{m}} \phi]=Y^{(n)}_{t}(\omega)[i_{q_{n}} \phi]$ for all $\phi \in \Phi$, that is $Y^{(m)}_{t}(\omega)=i'_{q_{n},q_{m}} Y^{(n)}_{t}(\omega)$. Hence, the property \emph{(2)(b)} of the statement of the Lemma is also satisfied. 

Finally, the property \emph{(3)} of the statement of the Lemma is a consequence of \eqref{equalityPaeXAndVersionYnOnPhi} and \eqref{defiSetOmegaY}. 
\end{prf}

\begin{proof}[Proof of Proposition \ref{propRegulTheoremCadlagContinuousVersionInBoundedInterval}]
We use similar arguments to those used by Mitoma in \cite{Mitoma:1983}. Let $\{q_{n}\}_{n \in \N}$, $\Omega_{Y}$ and $\{ Y^{(n)} \}_{n \in \N}$ be as given in Lemma \ref{lemmaRegulaTheo}. 

Let $\{ \varrho_{n} \}_{n \in \N}$ be an increasing sequence of continuous Hilbertian semi-norm on $\Phi$ such that for every $n \in \N$, $q_{n} \leq \varrho_{n}$ and $i_{q_{n}, \varrho_{n}}$ is Hilbert-Schmidt. Let $\vartheta$ be the countably Hilbertian topology on $\Phi$ generated by the semi-norms $\{ \varrho_{n}\}_{n \in \N}$. The topology $\vartheta$ is weaker than the nuclear topology on $\Phi$. 

Moreover, from Proposition \ref{propWeakerCountablyHilbertianTopology} we have that the completion $\widetilde{\Phi_{\vartheta}}$ of $\Phi_{\vartheta}$ satisfies  
\begin{equation} \label{dualCHTProofRegulTheo}
(\widetilde{\Phi_{\vartheta}})'= \bigcup_{n \in \N} \Phi'_{\varrho_{n}}.  
\end{equation}
Now, for every $n \in \N$ the map $i_{\varrho_{n}}$ is linear and continuous from $\Phi_{\vartheta}$ into $\Phi_{\varrho_{n}}$, and therefore $i_{\varrho_{n}}$ has a unique continuous and linear extension $\widetilde{i}_{\varrho_{n}}: \widetilde{\Phi_{\vartheta}} \rightarrow \Phi_{\varrho_{n}}$ (see \cite{Jarchow}, Theorem 2, Section 3.4, p.61-2). Hence, for each $n \in \N$ the dual map $(\widetilde{i}_{\varrho_{n}})': \Phi'_{\varrho_{n}} \rightarrow (\widetilde{\Phi_{\vartheta}})'_{\beta}$ is linear and continuous, and corresponds to the canonical inclusion from $\Phi'_{\varrho_{n}}$ into $(\widetilde{\Phi_{\vartheta}})'_{\beta}$. 

Let $Y=\{ Y_{t} \}_{t \in [0,T]}$ be defined for each $t \in [0,T]$ by
\begin{equation} \label{defiContVersionY}
Y_{t}(\omega) \defeq 
\begin{cases}
i'_{q_{n}, \varrho_{n}} Y_{t}^{(n)}(\omega), & \mbox{for } \omega \in \Omega_{Y}, \, n \geq N(\omega), \\
0, & \mbox{elsewhere}. 
\end{cases}
\end{equation}
For every $t \in [0,T]$, $Y_{t}$ is well-defined $(\widetilde{\Phi_{\vartheta}})'_{\beta}$-valued random variable. In effect, it is clear from \eqref{dualCHTProofRegulTheo} and \eqref{defiContVersionY} that $Y_{t}$ takes values in $(\widetilde{\Phi_{\vartheta}})'$. Now, as for all $m, n \in \N$, $m \geq n$, we have $q_{n} \leq q_{m}$, $q_{n} \leq \varrho_{n}$, and $q_{m} \leq \varrho_{m}$, then the following diagram commutes: 
\begin{equation} \label{commutDiagRegularVersion}
\begin{gathered}
\xymatrix{ 
\Phi'_{q_{n}} \ar[r]^{i'_{q_{n},\varrho_{n}}} \ar[d]_{i'_{q_{n},q_{m}}} & \Phi'_{\varrho_{n}} \ar[d]^{i'_{\varrho_{n},\varrho_{m}}} \\
\Phi'_{q_{m}} \ar[r]_{i'_{q_{m},\varrho_{m}}} & \Phi'_{\varrho_{m}}
}
\end{gathered}
\end{equation}
Then Lemma \ref{lemmaRegulaTheo}(2)(b) and \eqref{commutDiagRegularVersion} implies that for each $\omega \in \Omega_{Y}$, if $m \geq n \geq N(\omega)$ we have  
$$ i'_{q_{m}, \varrho_{m}} Y^{(m)}_{t}(\omega)= i'_{q_{m}, \varrho_{m}} \circ i'_{q_{n}, q_{m}} Y^{(n)}_{t}(\omega) =  i'_{\varrho_{n}, \varrho_{m}} \circ i'_{q_{n}, \varrho_{n}} Y^{(n)}_{t}(\omega), \quad \forall \, t \in [0,T]. $$  
Therefore, $Y_{t}$ is well-defined. Finally, the fact that $Y_{t}$ is $\mathcal{B}((\widetilde{\Phi_{\vartheta}})'_{\beta})/\mathcal{F}$-measurable map is a consequence of \eqref{defiContVersionY}, the fact that $Y^{(n)}_{t}$ is a $\Phi'_{q_{n}}$-valued random variable and that $i'_{q_{n},\varrho_{n}}$ is continuous, for every $n \in \N$. 

Now we are going to show that $Y$ is a $(\widetilde{\Phi_{\vartheta}})'_{\beta}$-valued continuous process. For every $n \in \N$, let $\{ \phi_{j}^{\varrho_{n}} \}_{j \in \N} \subseteq \Phi$ be a complete orthonormal system in $\Phi_{\varrho_{n}}$. Fix $\omega \in \Omega_{Y}$ and let $n \geq N(\omega)$. Then, from the definition of the dual operator $i'_{q_{n}, \varrho_{n}}$ of $i_{q_{n}, \varrho_{n}}$ and Lemma \ref{lemmaRegulaTheo}(1)-(2)(a), we have
\begin{eqnarray}
\sum_{j =1}^{\infty} \sup_{t \in [0,T]} \abs{ i'_{q_{n}, \varrho_{n}} Y_{t}^{(n)}(\omega) [\phi_{j}^{\varrho_{n}}]}^{2}   
& \leq & \sum_{j =1}^{\infty} \sup_{t \in [0,T]} q'_{n}(Y_{t}^{(n)}(\omega))^{2} q_{n}(i_{q_{n}, \varrho_{n}}\phi_{j}^{\varrho_{n}})^{2}   \nonumber  \\
& = & \left( \sup_{t \in [0,T]} q_{n}'(Y^{(n)}_{t}(\omega))^{2} \right) \norm{i_{q_{n}, \varrho_{n}}}^{2}_{\mathcal{L}_{2}(\Phi_{\varrho_{n}},\Phi_{q_{n}})} \nonumber \\ 
& < & \infty, \label{uniformBoundPathsRegularVersionY}
\end{eqnarray}
where $\norm{i_{q_{n}, \varrho_{n}}}_{\mathcal{L}_{2}(\Phi_{\varrho_{n}},\Phi_{q_{n}})} < \infty$ because $i_{q_{n},\varrho_{n}}$ is Hilbert-Schmidt. 

Next we prove the right continuity of the map $t \mapsto Y_{t}(\omega)$ in $\Phi'_{\varrho_{n}}$. Let $0 \leq t < T$. Then, from \eqref{defiContVersionY}, Parseval's identity, \eqref{uniformBoundPathsRegularVersionY}, the dominated convergence theorem and the continuity of each map $t \mapsto Y^{(n)}_{t}(\omega)[ i_{q_{n}, \varrho_{n}} \phi_{j}^{\varrho_{n}}]$ (Lemma \ref{lemmaRegulaTheo}(1)), we have 
\begin{eqnarray*}
\lim_{s \rightarrow t+} \varrho'_{n} ( Y_{t}(\omega) - Y_{s}(\omega))^{2} 
& = & \lim_{s \rightarrow t+} \sum_{j =1}^{\infty} \abs{ i'_{q_{n}, \varrho_{n}} (Y_{t}^{(n)}(\omega)-Y_{s}^{(n)}(\omega))[\phi_{j}^{\varrho_{n}}]}^{2} \\
& = & \sum_{j =1}^{\infty} \lim_{s \rightarrow t+} \abs{ Y_{t}^{(n)}(\omega)[i_{q_{n}, \varrho_{n}}  \phi_{j}^{\varrho_{n}}]-Y_{s}^{(n)}(\omega)[i_{q_{n}, \varrho_{n}}  \phi_{j}^{\varrho_{n}}]}^{2}  \\
& = & 0. 
\end{eqnarray*}
Therefore, the map $t \mapsto Y_{t}(\omega)$ is right-continuous in $\Phi'_{\varrho_{n}}$. The left continuity of $t \mapsto Y_{t}(\omega)$ in $\Phi'_{\varrho_{n}}$ can be proven similarly. Moreover, as the canonical inclusion $(\widetilde{i}_{\varrho_{n}})': \Phi'_{\varrho_{n}} \rightarrow (\widetilde{\Phi_{\vartheta}})'_{\beta}$ is continuous, then the continuity of $t \mapsto Y_{t}(\omega)$ in $\Phi'_{\varrho_{n}}$ implies that the map $t \mapsto (\widetilde{i}_{\varrho_{n}})' Y_{t}(\omega)$ is continuous from $[0,\infty)$ into $(\widetilde{\Phi_{\vartheta}})'_{\beta}$. Hence, $Y$ is a $(\widetilde{\Phi_{\vartheta}})'_{\beta}$-valued continuous process. 

Finally, the fact that for every $\phi \in \Phi$, $Y[\phi]$ is a version of $X(\phi)$ is a direct consequence of Lemma \ref{lemmaRegulaTheo}(3), \eqref{defiContVersionY} and because $\widehat{X}(\phi)$ is a version of $X(\phi)$. 

\end{proof}

\begin{proof}[Proof of Theorem \ref{theoRegularizationTheoremCadlagContinuousVersion}]
Let $\{ T_{k} \}_{k \in \N}$ an increasing sequence of positive numbers such that $\lim_{k \rightarrow \infty} T_{k}= \infty$. From Proposition \ref{propRegulTheoremCadlagContinuousVersionInBoundedInterval}, for every $k \in \N$ there exists a countably Hilbertian topology $\vartheta_{T_{k}}$ on $\Phi$ determined by an increasing sequence $\{ \varrho_{k,n} \}_{n \in \N}$ of continuous Hilbertian semi-norms on $\Phi$, and a $(\widetilde{\Phi_{\vartheta_{T_{k}}}})'_{\beta}$-valued continuous process $Y^{(T_{k})}= \{ Y^{(T_{k})}_{t} \}_{t \in [0,T_{k}]}$, such that for every $\phi \in \Phi$, $\{ Y^{(T_{k})}_{t}[\phi] \}_{t \in [0,T_{k}]}$ is a version of $\{ X_{t}(\phi) \}_{t \in [0,T_{k}]}$. 
%

Without loss of generality we can assume that for every $k \in \N$, $\varrho_{k,n} \leq \varrho_{k+1,n}$, for all $n \in \N$. This implies that for every $k \in \N$, the topology $\vartheta_{T_{k+1}}$ is weaker than $\vartheta_{T_{k}}$ and therefore that the canonical inclusion $i_{\vartheta_{T_{k}}, \vartheta_{T_{k+1}}}: \Phi_{\vartheta_{T_{k+1}}} \rightarrow \Phi_{\vartheta_{T_{k}}}$ is linear and continuous. Then, this map has a unique continuous and linear extension $\widetilde{i}_{\vartheta_{k}, \vartheta_{k+1}}: \widetilde{\Phi_{\vartheta_{T_{k+1}}}} \rightarrow \widetilde{\Phi_{\vartheta_{T_{k}}}}$. Hence, the dual map $(\widetilde{i}_{\vartheta_{T_{k}}, \vartheta_{T_{k+1}}})': (\widetilde{\Phi_{\vartheta_{T_{k}}}})'_{\beta} \rightarrow (\widetilde{\Phi_{\vartheta_{T_{k+1}}}})'_{\beta}$ is also linear and continuous, and corresponds to the canonical inclusion from $(\widetilde{\Phi_{\vartheta_{T_{k}}}})'_{\beta}$ into $(\widetilde{\Phi_{\vartheta_{T_{k+1}}}})'_{\beta}$. 

Let $\vartheta_{X}$ be the countably Hilbertian topology on $\Phi$ generated by the semi-norms $\{ \varrho_{k,n}: n, k \in \N\}$. Then, for each $k \in \N$ the topology $\vartheta_{T_{k}}$ is weaker than $\vartheta_{X}$. Hence the canonical inclusion $i_{\vartheta_{T_{k}},\vartheta_{X}}$ from $\Phi_{\vartheta_{X}}$ into $\Phi_{\vartheta_{T_{k}}}$ is linear and continuous. This map has an extension $\widetilde{i}_{\vartheta_{T_{k}},\vartheta_{X}}:\widetilde{\Phi_{\vartheta_{X}}} \rightarrow \widetilde{\Phi_{\vartheta_{T_{k}}}}$ that is also linear and continuous. The dual operator $(\widetilde{i}_{\vartheta_{T_{k}},\vartheta_{X}})':  (\widetilde{\Phi_{\vartheta_{T_{k}}}})'_{\beta} \rightarrow (\widetilde{\Phi_{\vartheta_{X}}})'_{\beta}$ corresponds to the  canonical inclusion from $(\widetilde{\Phi_{\vartheta_{T_{k}}}})'_{\beta}$ into $ (\widetilde{\Phi_{\vartheta_{X}}})'_{\beta}$ and is linear and continuous. For every $k \in \N$, one can easily verify that 
$\widetilde{i}_{\vartheta_{T_{k}},\vartheta_{X}}= \widetilde{i}_{\vartheta_{k}, \vartheta_{k+1}} \circ \widetilde{i}_{\vartheta_{T_{k+1}},\vartheta_{X}}$ and hence $(\widetilde{i}_{\vartheta_{T_{k}},\vartheta_{X}})'=(\widetilde{i}_{\vartheta_{T_{k}},\vartheta_{X}})' \circ (\widetilde{i}_{\vartheta_{T_{k}}, \vartheta_{T_{k+1}}})'$.  
 
Take $Y=\{ Y_{t} \}_{t \geq 0}$ defined by the prescription $Y_{t}=Y^{(T_{k})}_{t}$ if $t \in [0,T_{k}]$. Then $Y$ is a $(\widetilde{\Phi_{\vartheta_{X}}})'_{\beta}$-valued continuous process. To prove this, note that for any $k \in \N$, $\{ (\widetilde{i}_{\vartheta_{T_{k}}, \vartheta_{T_{k+1}}})' Y^{(T_{k})}_{t} \}_{t \in [0,T_{k}]}$ and $\{Y^{(T_{k+1})}_{t} \}_{t \in [0,T_{k}]}$ are continuous $(\widetilde{\Phi_{\vartheta_{T_{k+1}}}})'_{\beta}$-valued processes. Moreover, for every $\phi \in \Phi$ and $t \in T_{k}$ we have $\Prob$-a.e.  
\begin{align*}
(\, \widetilde{i}_{\vartheta_{T_{k}}, \vartheta_{T_{k+1}}})' \, Y^{(T_{k})}_{t}[ \, \widetilde{i}_{\vartheta_{T_{k+1}},\vartheta_{X}} \phi] & = Y^{(T_{k})}_{t}[ \, \widetilde{i}_{\vartheta_{T_{k}}, \vartheta_{T_{k+1}}} \circ \, \widetilde{i}_{\vartheta_{T_{k+1}},\vartheta_{X}} \phi] \\
{ } & =  Y^{(T_{k})}_{t}[ \, \widetilde{i}_{\vartheta_{T_{k}},\vartheta_{X}} \phi] = X_{t}(\phi) = Y^{(T_{k+1})}_{t} [ \, \widetilde{i}_{\vartheta_{T_{k+1}},\vartheta_{X}} \phi].
\end{align*}
Then, Proposition \ref{propCondiIndistingProcess} shows that $\{ (\widetilde{i}_{\vartheta_{T_{k}}, \vartheta_{T_{k+1}}})' Y^{(T_{k})}_{t} \}_{t \in [0,T_{k}]}$ and $\{Y^{(T_{k+1})}_{t} \}_{t \in [0,T_{k}]}$ are indistinguishable processes in $(\widetilde{\Phi_{\vartheta_{T_{k+1}}}})'_{\beta}$. Therefore, $Y$ is well-defined $ (\widetilde{\Phi_{\vartheta_{X}}})'_{\beta}$-valued process.

Now, because for each $k \in \N$ the $(\widetilde{\Phi_{\vartheta_{T_{k}}}})'_{\beta}$-valued process $\{ Y^{(T_{k})}_{t}\}_{ t \in [0,T_{k}]}$ is continuous and the map $(\widetilde{i}_{\vartheta_{T_{k}},\vartheta_{X}})'$ is continuous, it follows that $\{ Y_{t} \}_{t \in [0,T_{k}]}=\{ (\widetilde{i}_{\vartheta_{T_{k}},\vartheta_{X}})' \, Y^{(T_{k})}_{t}\}_{ t \in [0,T_{k}]}$ is a 
$ (\widetilde{\Phi_{\vartheta_{X}}})'_{\beta}$-valued continuous process. Therefore, $Y$ is a continuous processes in $ (\widetilde{\Phi_{\vartheta_{X}}})'_{\beta}$. 

Moreover, by the properties of the processes $Y^{(T_{k})}$ and the definition of $Y$ it follows that for every $\phi \in \Phi$, $Y[\phi]$ is a version of $X(\phi)$. Now, as the topology $\vartheta_{X}$ is weaker than the nuclear topology on $\Phi$, then $\Phi$ (equipped with the nuclear topology) is continuously embedded in $\Phi_{\theta_{X}}$, but as this last is continuously embedded in $\widetilde{\Phi_{\vartheta_{X}}}$, it follows that $\Phi$ is continuously embedded in $\widetilde{\Phi_{\vartheta_{X}}}$. Then, the space $ (\widetilde{\Phi_{\vartheta_{X}}})'_{\beta}$ is continuously embedded in $\Phi'_{\beta}$, and this implies that $Y$ is a $\Phi'_{\beta}$-valued, regular, continuous process that is a version of $X$. 

Finally, to prove the uniqueness note that if $Z$ is another version of $X$ satisfying the properties in the Theorem, then for every $\phi \in \Phi$ and $t \geq 0$, $Y_{t}[\phi]=X_{t}(\phi)=Z_{t}(\phi)$ $\Prob$-a.e., and because $Y$ and $Z$ are both continuous and regular $\Phi'_{\beta}$-valued processes it follows from Proposition \ref{propCondiIndistingProcess} that $Y$ and $Z$ are indistinguishable processes. 
\end{proof}

\begin{rema}
It follows from the proof of Theorem \ref{theoRegularizationTheoremCadlagContinuousVersion} (in particular from the proof of Proposition \ref{propRegulTheoremCadlagContinuousVersionInBoundedInterval}) that the $ (\widetilde{\Phi_{\vartheta_{X}}})'_{\beta}$-valued continuous (respectively c\`{a}dl\`{a}g) version $Y$ of $X$ can be chosen in such a way so that for every $\omega \in \Omega$ and $T>0$ there exists a continuous Hilbertian semi-norm $\varrho=\varrho(\omega,T)$ on $\Phi$ such that the map $t \mapsto Y_{t}(\omega)$ is continuous (respectively c\`{a}dl\`{a}g) from $[0,T]$ into the Hilbert space $\Phi'_{\varrho}$. 
\end{rema}

The assumption \emph{(2)} in Theorem \ref{theoRegularizationTheoremCadlagContinuousVersion} of equicontinuity for every $T>0$ of the family of cylindrical random variables $\{ X_{t} : t \in [0,T]\}$ was important to show that the map $\phi \mapsto \{ \widehat{X}_{t}(\phi) \}_{t \in [0,T]}$ from $\Phi$ into $C_{T}(\R)$ (respectively $D_{T}(\R)$) is continuous (Lemma \ref{lemmContiXAsMapToSpaceContiProcesses}). The next results shows that we can obtain the same conclusion without the above local equicontinuity property under the additional assumption that $\Phi$ is ultrabornological, i.e. that $\Phi$ is the inductive limit of Banach spaces (see \cite{NariciBeckenstein}, Theorem 13.2.11, p.448-9). Note that in the result below no nuclearity of the space is required.  

\begin{prop}\label{propContinuityXAsMapToSpaceContiRealProcessUltrabornological}
Let $\Psi$ be an ultrabornological space and let $X=\{X_{t} \}_{t \in [0,T]}$, be a cylindrical process in $\Psi'$ such that for each $\psi \in \Psi$, the real-valued process $\{ X_{t}(\psi) \}_{t \in [0,T]}$ has a continuous (respectively c\`{a}dl\`{a}g) version. Assume that for every $t \in [0,T]$ the map $X_{t}:  \Psi \rightarrow L^{0} \ProbSpace$ is continuous. Then, the linear mapping from $\Psi$ into $C_{T}(\R)$ (respectively $D_{T}(\R)$) given by $\psi \mapsto \{ X_{t}(\psi) \}_{t \in [0,T]}$ is continuous.
\end{prop} 
\begin{prf}
We can prove that the map $\psi \mapsto \{ X_{t}(\psi) \}_{t \in [0,T]}$ from $\Psi$ into $C_{T}(\R)$ (respectively $D_{T}(\R)$) is sequentially closed by following similar arguments to those used in the proof of Lemma \ref{lemmContiXAsMapToSpaceContiProcesses}. Then, as $\Psi$ is ultrabornological and $C_{T}(\R)$ (respectively $D_{T}(\R)$) is a complete, metrizable, topological vector space, the closed graph theorem (see \cite{Jarchow}, Proposition 2, Section 5.2 and Theorem 2, Section 5.4, p.90,94) shows that $\psi \mapsto \{ X_{t}(\psi) \}_{t \in [0,T]}$ is continuous.    
\end{prf}

We have the following version of the regularization theorem for cylindrical processes in the dual of an ultrabornological nuclear space. 

\begin{coro} \label{coroRegulTheoUltrabornologicalSpace}
Let $\Phi$ be an ultrabornological nuclear space. Let $X=\{X_{t} \}_{t \geq 0}$ be a cylindrical process in $\Phi'$ satisfying:
\begin{enumerate}
\item For each $\phi \in \Phi$, the real-valued process $X(\phi)=\{ X_{t}(\phi) \}_{t \geq 0}$ has a continuous (respectively c\`{a}dl\`{a}g) version.
\item For every $t \geq 0$, $X_{t}:\Phi \rightarrow L^{0} \ProbSpace$ is continuous.
\end{enumerate}
Then, there exists a countably Hilbertian topology $\vartheta_{X}$ on $\Phi$ 
and a $(\widetilde{\Phi_{\vartheta_{X}}})'_{\beta}$-valued continuous (respectively c\`{a}dl\`{a}g) process $Y= \{ Y_{t} \}_{t \geq 0}$, such that for every $\phi \in \Phi$, $Y[\phi]= \{ Y_{t}[\phi] \}_{t \geq 0}$ is a version of $X(\phi)= \{ X_{t}(\phi) \}_{t \geq 0}$. Moreover, $Y$ is a $\Phi'_{\beta}$-valued, regular, continuous (respectively c\`{a}dl\`{a}g) version of $X$ that is unique up to indistinguishable versions. 
\end{coro}
\begin{prf}
Let $T>0$. With the same terminology as in the proof of Proposition \ref{propRegulTheoremCadlagContinuousVersionInBoundedInterval}, it follows from Proposition \ref{propContinuityXAsMapToSpaceContiRealProcessUltrabornological} that the linear mapping from $\Phi$ into $C_{T}(\R)$ (respectively $D_{T}(\R)$) given by $\psi \mapsto \{ \widehat{X}_{t}(\phi) \}_{t \in [0,T]}$ is continuous. We can use this result to prove Lemma \ref{lemmPnContiCharacFunctSupremum}. Then, the same arguments used in the remain of the proof of Proposition \ref{propRegulTheoremCadlagContinuousVersionInBoundedInterval} show that the conclusions of Proposition \ref{propRegulTheoremCadlagContinuousVersionInBoundedInterval} are still valid in our present context. Therefore, the proof of Theorem \ref{theoRegularizationTheoremCadlagContinuousVersion} can be replicated without any change and showing the existence of the countably Hilbertian topology $\theta_{X}$ and the $(\widetilde{\Phi_{\vartheta_{X}}})'_{\beta}$-valued process $Y$ satisfying the conclusions of the corollary. 
\end{prf}

From the above result we obtain the following version of the regularization theorem for stochastic processes taking values in the dual of an ultrabornological nuclear space.

\begin{coro} \label{coroRegulTheoUltrabornologicalSpaceStochastProcess}
Let $\Phi$ be an ultrabornological nuclear space. Let $X=\{X_{t} \}_{t \geq 0}$ be a $\Phi'_{\beta}$-valued process satisfying:
\begin{enumerate}
\item For each $\phi \in \Phi$, the real-valued process $X[\phi]=\{ X_{t}[\phi] \}_{t \geq 0}$ has a continuous (respectively c\`{a}dl\`{a}g) version.
\item For every $t \geq 0$, the distribution $\mu_{X_{t}}$ of $X_{t}$ is a Radon probability measure.
\end{enumerate}
Then, there exists a countably Hilbertian topology $\vartheta_{X}$ on $\Phi$ 
and a $(\widetilde{\Phi_{\vartheta_{X}}})'_{\beta}$-valued continuous (respectively c\`{a}dl\`{a}g) process $Y= \{ Y_{t} \}_{t \geq 0}$, such that for every $\phi \in \Phi$, $Y[\phi]= \{ Y_{t}[\phi] \}_{t \geq 0}$ is a version of $X[\phi]= \{ X_{t}[\phi] \}_{t \geq 0}$. Moreover, $Y$ is a $\Phi'_{\beta}$-valued, regular, continuous (respectively c\`{a}dl\`{a}g) version of $X$ that is unique up to indistinguishable versions. 
\end{coro}
\begin{prf}
First, because the space $\Phi$ is ultrabornological it is also barreled (see \cite{NariciBeckenstein}, p.449). Hence, it follows from the assumption \emph{(2)} in the Corollary and from Theorem \ref{theoCharacterizationRegularRV} that for each $t \geq 0$ the map $X_{t}:\Phi \rightarrow L^{0} \ProbSpace$, $\phi \mapsto X_{t}[\phi]$ is continuous. Then, the existence of the countably Hilbertian topology $\theta_{X}$ and the $(\widetilde{\Phi_{\vartheta_{X}}})'_{\beta}$-valued process $Y$ satisfying the conclusions of the corollary follows from Corollary \ref{coroRegulTheoUltrabornologicalSpace}.  
\end{prf}

We finish this section with some comparisons of our results with those obtained by other authors. 

First, it is clear that Theorem \ref{theoRegularizationTheoremCadlagContinuousVersion} is a generalization of the regularization theorem of It\^{o} and Nawata (Theorem \ref{regularizationTheorem}).  

Now, if $\Phi$ is a nuclear space that is also a Fr\'{e}chet space or the countable inductive limit of Fr\'{e}chet spaces, then $\Phi$ is an ultrabornological space (see \cite{Jarchow}, Corollaries 4 and 5, Section 13.1, p.273). In this case if $X=\{X_{t} \}_{t \geq 0}$ is a $\Phi'_{\beta}$-valued process then for every $t \geq 0$ we have that $\mu_{X_{t}}$ is a Radon probability measure (see Section \ref{subSectionCylAndStocProcess}). Then, if for each $\phi \in \Phi$, the real-valued process $X[\phi]=\{ X_{t}[\phi] \}_{t \geq 0}$ has a continuous (respectively c\`{a}dl\`{a}g) version, then Corollary \ref{coroRegulTheoUltrabornologicalSpaceStochastProcess} shows that $X$ has a $\Phi'_{\beta}$-valued continuous (respectively c\`{a}dl\`{a}g) version. Therefore, Corollary \ref{coroRegulTheoUltrabornologicalSpaceStochastProcess} (and consequently Theorem \ref{theoRegularizationTheoremCadlagContinuousVersion}) generalizes the results obtained by Mitoma \cite{Mitoma:1983}, Fouque \cite{Fouque:1984} and Fernique \cite{Fernique:1989}. 

On the other hand, Martias \cite{Martias:1988} showed a version of the regularization theorem under the assumptions that $\Phi$ is a separable nuclear space and that $X$ is of the form $X:[0,T] \times \Omega \rightarrow \Phi'$, such that for each $\phi \in \Phi$, $\{ X_{t}[\phi] \}_{t \in [0,T]}$ is a real-valued process with a continuous (respectively c\`{a}dl\`{a}g) version and also assuming that the map $\phi \mapsto \{ X_{t}[\phi] \}_{t \in [0,T]}$ from $\Phi$ into $C_{T}(\R)$ (respectively $D_{T}(\R)$) is continuous. Note that the assumptions on $\Phi$ and $X$ in Theorem \ref{theoRegularizationTheoremCadlagContinuousVersion} are weaker than the assumptions used by Martias. In particular, one can easily check that the continuity of the map $\phi \mapsto \{ X_{t}[\phi] \}_{t \in [0,T]}$ implies assumption \emph{(2)} in Theorem \ref{theoRegularizationTheoremCadlagContinuousVersion}. Hence, Theorem \ref{theoRegularizationTheoremCadlagContinuousVersion} constitutes a generalization of the result obtained by Martias.

\section{Existence of Continuous and C\`{a}dl\`{a}g Versions in a Hilbert Space Continuously Embedded in $\Phi'_{\beta}$}\label{sectionContCadVersionHilbertSpace}

In this section we introduce conditions for the existence of continuous and c\`{a}dl\`{a}g versions taking values in a Hilbert space $\Phi'_{\varrho}$, where $\varrho$ is a continuous Hilbertian semi-norm on $\Phi$. We start with the following result that specializes the Regularization Theorem. 
 
\begin{theo} \label{theoRegularTheoHilbertSpaceCadlagContinuousVersion}
Let $X=\{X_{t} \}_{t \geq 0}$ be a cylindrical process in $\Phi'$ satisfying:
\begin{enumerate}
\item For each $\phi \in \Phi$, the real-valued process $X(\phi)=\{ X_{t}(\phi) \}_{t \geq 0}$ has a continuous (respectively c\`{a}dl\`{a}g) version.
\item There exists a continuous Hilbertian semi-norm $p$ on $\Phi$ such that for every $t \geq 0$, $X_{t}:\Phi \rightarrow L^{0} \ProbSpace$ is $p$-continuous.  
\end{enumerate}
Then, there exists a continuous Hilbertian semi-norm $\varrho$ on $\Phi$, $p \leq \varrho$, such that $i_{p,\varrho}$ is Hilbert-Schmidt and a $\Phi'_{\varrho}$-valued continuous (respectively c\`{a}dl\`{a}g) process $Y= \{ Y_{t} \}_{t \geq 0}$, such that for every $\phi \in \Phi$, $Y[\phi]= \{ Y_{t}[\phi] \}_{t \geq 0}$ is a version of $X(\phi)= \{ X_{t}(\phi) \}_{t \geq 0}$. Moreover, $Y$ is unique up to indistinguishable versions in $\Phi'_{\beta}$. 
\end{theo}
\begin{prf}
Let $T>0$. With the same terminology as in the proof of Proposition \ref{propRegulTheoremCadlagContinuousVersionInBoundedInterval}, the assumption \emph{(2)} of the Theorem implies that for every $t \in [0,T]$ the map $\widehat{X}_{t}: \Phi \rightarrow L^{0} \ProbSpace$ is $p$-continuous. Then, for all $t \in [0,T]$ the map $\widehat{X}_{t}$ has a continuous and linear extension $\widetilde{X}_{t}:\Phi_{p} \rightarrow L^{0} \ProbSpace$. Hence, from similar arguments to those in the proof of Lemma \ref{lemmContiXAsMapToSpaceContiProcesses} we can show that the linear mapping $\widetilde{X}:\Phi_{p} \rightarrow C_{T}(\R)$, $\phi \mapsto \widetilde{X}(\phi)=\{ \widetilde{X}_{t}(\phi) \}_{t \in [0,T]}$ is continuous. This in turns implies that the map $\widehat{X}:\Phi \rightarrow C_{T}(\R)$, $\phi \mapsto \widehat{X}(\phi)=\{ \widehat{X}_{t}(\phi) \}_{t \in [0,T]}$ is $p$-continuous. 
Hence, in the proof of Lemma \ref{lemmaRegulaTheo}, we have that \eqref{continuityOrigenCHTRegularVersionOperator} is satisfied for $p_{n}=p$, for all $n \in \N$. In that case, we have in \eqref{defiSetOmegaNRegulaTheo}, \eqref{defiSetGammaNRegulTheo} and \eqref{defiSetANRegulTheo} that $q_{n}=q$ for some continuous Hilbertian semi-norm $q$ on $\Phi$, $p \leq q$, and such that $i_{p,q}$ is Hilbert-Schmidt. Therefore, we have from \eqref{defiSetLambdaNRegulTheo} and \eqref{defiVersionYnRegularTheorem} that $\Lambda_{n}=\Lambda_{m}$ and $Y^{(n)}=Y^{(m)}$ for all $m,n \in \N$. If we choose $\varrho$, $q \leq \varrho$, such that $i_{q,\varrho}$ is Hilbert-Schmidt, and if in the proof of Proposition \ref{propRegulTheoremCadlagContinuousVersionInBoundedInterval} we take $\varrho_{n}=\varrho$ for all $n \in \N$, then $Y=\{Y_{t}\}_{t \in [0,T]}$ defined by \eqref{defiContVersionY} is a $\Phi'_{\varrho}$-valued continuous processes such that for every $\phi \in \Phi$, $Y[\phi]=\{Y_{t}[\phi]\}_{t \in [0,T]}$ is a version of $X(\phi)=\{X_{t}(\phi)\}_{t \in [0,T]}$. 

Let $\{ T_{k} \}_{k \in \N}$ an increasing sequence of positive numbers such that $\lim_{k \rightarrow \infty} T_{k}= \infty$. Then, as proved on the above paragraph for each $k \in \N$ there exists a $\Phi'_{\varrho}$-valued continuous process $Y^{(k)}= \{ Y^{(k)}_{t} \}_{t \in [0,T_{k}]}$, such that for every $\phi \in \Phi$,  $\{ Y^{(k)}_{t}[\phi] \}_{t \in [0,T_{k}]}$ is a version of $\{ X_{t}(\phi)\}_{t \in [0,T_{k}]}$.
 
Take $Y=\{ Y_{t} \}_{t \geq 0}$ defined by the prescription $Y_{t}=Y^{(k)}_{t}$ if $t \in [0,T_{k}]$. Note that $Y$ is a well-defined $\Phi'_{\varrho}$-valued process. To prove this, observe that for each each $k \in \N$, $Y^{(k)}$ and $Y^{(k+1)}$ are continuous processes such that for every $\phi \in \Phi$, $Y^{(k)}_{t}[\phi]=X_{t}(\phi)=Y^{(k+1)}_{t}[\phi]$ $\Prob$-a.e. for all $t \in [0,T_{k}]$, then Proposition \ref{propCondiIndistingProcess} shows that $\{ Y^{(k)}_{t} \}_{ t \in [0,T_{k}]}$ and $\{ Y^{(k+1)}_{t} \}_{ t \in [0,T_{k}]}$ are indistinguishable processes. Therefore, $Y$ is a $\Phi'_{\varrho}$-valued continuous process such that for every $\phi \in \Phi$, $Y[\phi]$ is a version of $X(\phi)$. Moreover, from Proposition \ref{propCondiIndistingProcess} again it follows that $Y$ is unique up to indistinguishable versions in $\Phi'_{\beta}$. 
\end{prf}

Now we consider conditions for the existence of continuous or c\`{a}dl\`{a}g version such that in a given bounded interval of time they take values and have uniform finite moments in some Hilbert space $\Phi'_{\varrho}$. For the proof we will need the following terminology. For $n \in \N$, we denote by $C^{n}_{T}(\R)$ (respectively $D^{n}_{T}(\R)$) the linear space of all the continuous (respectively c\`{a}dl\`{a}g) processes satisfying $\Exp  \sup_{t \in [0,T]} \abs{Z_{t}}^{n}< \infty$. It is a Banach space equipped with the norm $\norm{Z}_{n,T}= \left( \Exp  \sup_{t \in [0,T]} \abs{Z_{t}}^{n} \right)^{1/n}$.   

\begin{theo} \label{theoExistenceCadlagContVersionHilbertSpaceFiniteMoments}
Let $X=\{X_{t} \}_{t \geq 0}$ be a cylindrical process in $\Phi'$ satisfying:
\begin{enumerate}
\item For each $\phi \in \Phi$, the real-valued process $X(\phi)=\{ X_{t}(\phi) \}_{t \geq 0}$ has a continuous (respectively c\`{a}dl\`{a}g) version.
\item For every $T > 0$, the family $\{ X_{t}: t \in [0,T] \}$ of linear maps from $\Phi$ into $L^{0} \ProbSpace$ is equicontinuous.  
\item There exists $n \in \N$ such that $\forall \, T>0$, $\Exp \left( \sup_{t \in [0,T]} \abs{X_{t}[\phi]}^{n} \right) < \infty$, $\forall \, \phi \in \Phi$. 
\end{enumerate}
Then, there exists a countably Hilbertian topology $\vartheta_{X}$ on $\Phi$ determined by an increasing sequence $\{ q_{k} \}_{k \in \N}$ of continuous Hilbertian semi-norms on $\Phi$ and a $(\widetilde{\Phi_{\vartheta_{X}}})'_{\beta}$-valued continuous (respectively c\`{a}dl\`{a}g) process $Y= \{ Y_{t} \}_{t \geq 0}$, satisfying:
\begin{enumerate}[label=(\alph*)]
\item For every $\phi \in \Phi$, $Y[\phi]= \{ Y_{t}[\phi] \}_{t \geq 0}$ is a version of $X(\phi)= \{ X_{t}(\phi) \}_{t \geq 0}$, 
\item For every $T>0$, there exists $k \in \N$ such that $\{ Y_{t} \}_{t \in [0,T]}$ is a $\Phi'_{q_{k}}$-valued continuous (respectively c\`{a}dl\`{a}g) process satisfying $\Exp \left( \sup_{t \in [0,T]} q_{k}'(Y_{t})^{n} \right) < \infty$.  
\end{enumerate} 
Furthermore, $Y$ is a $\Phi'_{\beta}$-valued, regular, continuous (respectively c\`{a}dl\`{a}g) version of $X$ that is unique up to indistinguishable versions.
\end{theo}
\begin{prf} We prove the continuous case as the c\`{a}dl\`{a}g case follows from similar arguments. We do this in two steps. 

\textbf{Step 1} We are going to prove the following: \emph{For every $T>0$ there exists a continuous Hilbertian semi-norm $q$ on $\Phi$ and a $\Phi'_{q}$-valued continuous process $Y=\{ Y_{t} \}_{t \in [0,T]}$, such that for every $\phi \in \Phi$, $Y[\phi]=\{ Y_{t}[\phi] \}_{t \in [0,T]}$ is a version of $X(\phi)=\{ X_{t}(\phi) \}_{t \in [0,T]}$ and $\Exp \left( \sup_{t \in [0,T]} q'(Y_{t})^{n} \right)< \infty$. }

Let $T>0$ and define $p: \Phi \rightarrow [0,\infty)$ by 
\begin{equation} \label{defiSemiNormPUniformMoments}
p(\phi) = \left[ \Exp \left( \sup_{t \in [0,T]} \abs{X_{t}(\phi)}^{n} \right) \right]^{1/n}, \quad \forall \, \phi \in \Phi.
\end{equation}
We proceed to show that $p$ is a continuous semi-norm on $\Phi$. First, observe that the map from $\Phi$ into $C^{n}_{T}(\R)$ given by $\phi \mapsto X(\phi)=\{ X_{t}(\phi) \}_{t \in [0,T]}$ is linear. Then, because $\norm{ \cdot }_{n,T}$ its a norm on $C^{n}_{T}(\R)$ and $p(\phi)=\norm{X(\phi)}_{n,T}$ for all $\phi \in \Phi$, it follows that $p$ is a semi-norm.  

To prove that $p$ is continuous, one can use the closed graph theorem and similar arguments to those used in the proof of Lemma  \ref{lemmContiXAsMapToSpaceContiProcesses} to show that the map $\phi \mapsto X(\phi)$ from $\Phi$ into $C^{n}_{T}(\R)$ is continuous. But because $p(\phi)=\norm{X(\phi)}_{n,T}$, for all $\phi \in \Phi$, then the continuity of the maps $\phi \mapsto X(\phi)$ and $Z \mapsto \norm{Z}_{n,T}$ implies that $p$ is continuous. 

Now, from the Markov and Jensen inequalities, for any $\epsilon >0$ we have
$$ \Prob \left( \omega: \sup_{t \in [0,T]} \abs{X_{t}(\phi)(\omega)} > \epsilon \right) \leq \frac{1}{\epsilon}   
\left[ \Exp \left( \sup_{t \in [0,T]} \abs{X_{t}(\phi)}^{n} \right) \right]^{1/n} = \frac{1}{\epsilon} p(\phi), \quad \forall \phi \in \Phi.$$
Therefore, the map $X: \Phi \rightarrow C_{T}(\R)$ given by $ \psi \mapsto X(\phi)= \{ X_{t}(\phi) \}_{t \in [0,T]}$ is $p$-continuous. Now, as the topology on $\Phi$ is generated by a family of Hilbertian semi-norms and because the semi-norm $p$ is continuous, there exists a continuous Hilbertian semi-norm $\varrho$ on $\Phi$ such that $p \leq \varrho$. This implies that the map $X: \Phi \rightarrow C_{T}(\R)$ is $\varrho$-continuous. Then, it follows from Theorem  \ref{theoRegularTheoHilbertSpaceCadlagContinuousVersion} that there exists a continuous Hilbertian semi-norm $r$ on $\Phi$, $\varrho \leq r$ such that $i_{r,\varrho}$ is Hilbert-Schmidt and such that $X$ has a $\Phi'_{r}$-valued continuous version $\widehat{Y}= \{ \widehat{Y}_{t} \}_{t \in [0,T]}$.  

Let $q$ be a continuous Hilbertian semi-norm on $\Phi$ such that $r \leq q$ and $i_{r,q}$ is Hilbert-Schmidt. Then, the dual operator $i'_{r,q}:\Phi'_{r} \rightarrow \Phi'_{q}$ is Hilbert-Schmidt and hence is $n$-summing (see \cite{DiestelJarchowTonge},  Corollary 4.13, p.85). Then, from the Pietsch domination theorem (see \cite{DiestelJarchowTonge}, Theorem 2.12, p.44) there exists a constant $C>0$, and a Radon probability measure $\nu$ on the unit ball $B^{*}_{r}(1)$ of $\Phi_{r}$ (equipped with the weak topology) such that, 
\begin{equation}\label{integralInequalitySummingOperator}
q'(i'_{r,q} f ) \leq C \cdot \left( \int_{B^{*}_{r}(1)} \abs{f[\phi]}^{n} \nu(d\phi) \right)^{1/n}, \quad \forall \, f \in \Phi'_{r}.    
\end{equation}
As $\widehat{Y}$ is a $\Phi'_{r}$-valued continuous process, then $\phi \mapsto \widehat{Y}[\phi]=\{\widehat{Y}_{t}[\phi]\}_{t \in [0,T]}$ is a continuous and linear map from $\Phi_{r}$ into $C^{n}_{T}(\R)$. Therefore, it follows from \eqref{integralInequalitySummingOperator} that, 
\begin{eqnarray*}
\Exp \left( \sup_{t \in [0,T]}  q'(i'_{r,q} \widehat{Y}_{t})^{n} \right) & \leq & C^{n} \,  \Exp \sup_{t \in [0,T]}  \int_{B^{*}_{r}(1)} \abs{\widehat{Y}_{t}[\phi]}^{n} \nu(d\phi) \\
& \leq & C^{n}  \int_{B^{*}_{r}(1)} \norm{\widehat{Y}[\phi]}_{c,n,T}^{n} \, \nu(d\phi) \\
& \leq & C^{n} \norm{\widehat{Y}}_{\mathcal{L}(\Phi_{r},C^{n}_{T}(\R))}^{n} < \infty. 
\end{eqnarray*}
Hence, $Y=\{ Y_{t} \}_{t \in [0,T]}$, defined by $Y_{t}=i'_{r,q} \widehat{Y}_{t}$ for every $t \in [0,T]$, is a $\Phi'_{q}$-valued continuous process such that for every $\phi \in \Phi$, $Y[\phi]=\{ Y_{t}[\phi] \}_{t \in [0,T]}$ is a version of $X(\phi)=\{ X_{t}(\phi) \}_{t \in [0,T]}$ and such that $\Exp \left( \sup_{t \in [0,T]} q'(Y_{t})^{n} \right)< \infty$. 

\textbf{Step 2}
Let $\{ T_{k} \}_{k \in \N}$ an increasing sequence of positive numbers such that $\lim_{k \rightarrow \infty} T_{k}= \infty$. Then, it follows from Step 1 that for each $k \in \N$ there exists a continuous Hilbertian semi-norm $q_{k}$ on $\Phi$ and a $\Phi'_{q_{k}}$-valued continuous process $Y^{(k)}=\{ Y^{(k)}_{t} \}_{t \in [0,T_{k}]}$, such that for every $\phi \in \Phi$, $Y^{(k)}[\phi]=\{ Y^{(k)}_{t}[\phi] \}_{t \in [0,T_{k}]}$ is a version of $X(\phi)=\{ X_{t}(\phi) \}_{t \in [0,T_{k}]}$ and $\Exp \left( \sup_{t \in [0,T_{k}]} q'_{k}(Y^{(k)}_{t})^{n} \right)< \infty$. 

Without loss of generality we can assume that the sequence of semi-norms $\{ q_{k} \}_{k \in \N}$ is increasing. 
Let $\vartheta_{X}$ be the countably Hilbertian topology on $\Phi$ generated by the semi-norms $\{ q_{k} \}_{k \in \N}$. Take $Y=\{ Y_{t} \}_{t \geq 0}$ defined by the prescription $Y_{t}= Y^{(k)}_{t}$ if $t \in [0,T_{k}]$. In a similar way as we did in the proof of Theorem \ref{theoRegularizationTheoremCadlagContinuousVersion}, from the corresponding properties of the processes $Y^{(k)}$s  we can conclude that $Y$ is a $(\widetilde{\Phi_{\vartheta_{X}}})'_{\beta}$-valued continuous process satisfying the properties \emph{(a)} and \emph{(b)} of the Theorem. 
Finally, as in the proof of Theorem \ref{theoRegularizationTheoremCadlagContinuousVersion} we can show that as a $\Phi'_{\beta}$-valued process $Y$ is unique up to indistinguishable versions. 
\end{prf}

The following result is an specialized version of Theorem \ref{theoExistenceCadlagContVersionHilbertSpaceFiniteMoments} that establish conditions for the existence of continuous or c\`{a}dl\`{a}g version in single Hilbert space $\Phi'_{\varrho}$ with  uniform finite moments in every bounded interval of time. 
 
\begin{theo} \label{theoExistenceCadlagContVersionHilbertSpaceUniformBoundedMoments}
Let $X=\{X_{t} \}_{t \geq 0}$ be a cylindrical process in $\Phi'$ satisfying:
\begin{enumerate}
\item For each $\phi \in \Phi$, the real-valued process $X(\phi)=\{ X_{t}(\phi) \}_{t \geq 0}$ has a continuous (respectively c\`{a}dl\`{a}g) version.
\item There exists $n \in \N$ and a continuous Hilbertian semi-norm $\varrho$ on $\Phi$ such that for all $T>0$ there exists $C(T)>0$ such that 
\begin{equation} \label{uniformBoundMomentsByHilbertSeminorm}
\Exp \left( \sup_{t \in [0,T]} \abs{X_{t}(\phi)}^{n} \right) \leq C(T) \varrho(\phi)^{n}, \quad \forall \, \phi \in \Phi.
\end{equation} 
\end{enumerate}
Then, there exists a continuous Hilbertian semi-norm $q$ on $\Phi$, $\varrho \leq q$, such that $i_{\varrho,q}$ is Hilbert-Schmidt and there exists a $\Phi'_{q}$-valued continuous (respectively c\`{a}dl\`{a}g) process, satisfying:
\begin{enumerate}[label=(\alph*)]
\item For every $\phi \in \Phi$, $Y[\phi]= \{ Y_{t}[\phi] \}_{t \geq 0}$ is a version of $X(\phi)= \{ X_{t}(\phi) \}_{t \geq 0}$, 
\item For every $T>0$, $\Exp \left( \sup_{t \in [0,T]} q'(Y_{t})^{n} \right) < \infty$.   
\end{enumerate} 
Furthermore, $Y$ is a $\Phi'_{\beta}$-valued continuous (respectively c\`{a}dl\`{a}g) version of $X$ that is unique up to indistinguishable versions.
\end{theo}
\begin{prf} As before, we will prove the continuous case as the c\`{a}dl\`{a}g case follows similarly.
 
Fix $T>0$. Note that \eqref{uniformBoundMomentsByHilbertSeminorm} is equivalent to $p \leq C(T)^{1/n} \varrho$, where $p$ is the semi-norm defined in \eqref{defiSemiNormPUniformMoments}. Therefore, the map from $\Phi$ into $C_{T}(\R)$ given by $\phi \mapsto X(\phi)=\{ X_{t}(\phi)\}_{t \in [0,T]}$ is $\varrho$-continuous. Hence, similar arguments to those in Step 1 of the proof of Theorem  \ref{theoExistenceCadlagContVersionHilbertSpaceFiniteMoments} show that there exists a continuous Hilbertian semi-norm $q$ on $\Phi$ (only depending on $\varrho$), $\varrho \leq q$, such that $i_{\varrho,q}$ is Hilbert-Schmidt, and a $\Phi'_{q}$-valued continuous version $Y= \{ Y_{t} \}_{t \in [0,T]}$ of $\{ X_{t} \}_{t \in [0,T]}$ satisfying $\Exp \left( \sup_{t \in [0,T]} q'(Y_{t})^{n} \right) < \infty$. 

Let $\{ T_{k} \}_{k \in \N}$ an increasing sequence of positive numbers such that $\lim_{k \rightarrow \infty} T_{k}= \infty$. Then, as proved in the above paragraph for each $k \in \N$ there exists a $\Phi'_{q}$-valued continuous version $Y^{(k)}= \{ Y^{(k)}_{t} \}_{t \in [0,T_{k}]}$ of $\{ X_{t} \}_{t \in [0,T_{k}]}$ that satisfies $\Exp \left( \sup_{t \in [0,T_{k}]} q'(Y^{(k)}_{t})^{n} \right) < \infty$. Observe that for each $k \in \N$, the $\Phi'_{q}$-valued continuous processes $\{ Y^{(k)}_{t} \}_{t \in [0,T_{k}]}$ and $\{ Y^{(k+1)}_{t} \}_{t \in [0,T_{k}]}$ are indistinguishable. Therefore, if we define $Y=\{ Y_{t} \}_{t \geq 0}$ by the prescription $Y_{t}=Y^{(k)}_{t}$ if $t \in [0,T_{k}]$, then $Y$ is a $\Phi'_{q}$-valued continuous version of $X$ satisfying $\Exp \left( \sup_{t \in [0,T]} q'(Y_{t})^{n} \right) < \infty$, for all $T>0$.       
\end{prf}

\section{Applications to Cylindrical Martingales in $\Phi'$} \label{sectionApplicaStochProcesses}

In this section we apply the results in Section \ref{sectionContCadVersionHilbertSpace} to study some properties of cylindrical martingales in $\Phi'$. Consider a filtration $\left\{ \mathcal{F}_{t} \right\}_{t \geq 0}$ on $\ProbSpace$ that satisfies the \emph{usual conditions}, i.e. it is right continuous and $\mathcal{F}_{0}$ contains all sets of $\mathcal{F}$ of $\Prob$-measure zero. 

Let $n \geq 1$. Denote by $\mathcal{M}^{n}(\R)$ the space of all the real-valued martingales $M=\{ M_{t} \}_{t \geq 0}$ satisfying 
\begin{equation} \label{defiSemiNormSpaceBoundedMartingales}
\norm{ M}_{\mathcal{M}^{n}(\R)}= \left[ \Exp \left( \sup_{t \geq 0} \abs{M_{t}}^{n} \right) \right]^{1/n} < \infty. 
\end{equation}
It is clear that $\norm{ \cdot }_{\mathcal{M}^{n}(\R)}$ defines a norm on $\mathcal{M}^{n}(\R)$. Moreover, $( \mathcal{M}^{n}(\R), \norm{ \cdot }_{\mathcal{M}^{n}(\R)})$ is a Banach space (see \cite{DellacherieMeyer}). Note that if the real-valued martingale $M=\{ M_{t} \}_{t \geq 0}$ satisfies $\sup_{t \geq 0} \Exp \left( \abs{M_{t}}^{n} \right)< \infty$, for $n >2$, then Doob's inequality shows that $\norm{ M}_{\mathcal{M}^{n}(\R)}< \infty$ and hence $M \in \mathcal{M}^{n}(\R)$.  

\begin{defi}
A \emph{cylindrical martingale} in $\Phi'$ is a cylindrical process $M=\{ M_{t} \}_{t \geq 0}$ such that for each $\phi \in \Phi$, the real-valued process $M(\phi)= \{ M_{t}(\phi) \}_{t \geq 0}$ is a $\{\mathcal{F}_{t} \}$-adapted martingale. 
\end{defi} 

The following theorem contains some of the basic properties of cylindrical martingales in $\Phi'$.

\begin{theo}\label{propertiesMartingales} 
Let $M=\left\{ M_{t} \right\}_{t\geq 0}$ be a cylindrical martingale in $\Phi'$ such that for each $t \geq 0$ the map $M_{t}: \Phi \rightarrow L^{0} \ProbSpace$ is continuous. Then, $M$ has a $\Phi'_{\beta}$-valued \cadlag version $\widetilde{M}=\{ \widetilde{M}_{t} \}_{t\geq 0}$. Moreover, we have the following:
\begin{enumerate}
\item If $M$ is $n$-th integrable, for $n \geq 2$, then for each $T>0$ there exists a continuous Hilbertian semi-norm $q_{T}$ on $\Phi$ such that $\{ \widetilde{M}_{t} \}_{t \in [0, T] }$ is a $\Phi'_{q_{T}}$-valued \cadlag martingale satisfying $\Exp \left( \sup_{t \in [0, T] } q'_{T} ( \widetilde{M}_{t} )^{n} \right) < \infty$.  
\item Moreover, if for $n \geq 2$, $\sup_{t \geq 0} \Exp \left( \abs{M_{t}(\phi)}^{n} \right)< \infty$ for each $\phi \in \Phi$, then there exists a continuous Hilbertian semi-norm $q$ on $\Phi$ such that $\widetilde{M}$ is a $\Phi'_{q}$-valued c\`{a}dl\`{a}g martingale satisfying $\Exp \left( \sup_{t \geq 0}  q'( \widetilde{M}_{t} )^{n} \right) < \infty$.    
\end{enumerate}
If for each $ \phi \in \Phi$ the real-valued process $\{ M_{t}(\phi) \}_{t \geq 0}$ has a continuous version, then $\widetilde{M}$ can be chosen to be continuous and such that it satisfies \emph{(1)}$-$\emph{(2)} above replacing the property c\`{a}dl\`{a}g by continuous. 
\end{theo}
\begin{prf} To show the existence of the $\Phi'_{\beta}$-valued \cadlag version $\widetilde{M}$ of $M$ we need to check that the conditions in Theorem \ref{theoRegularizationTheoremCadlagContinuousVersion} are satisfied. As for each $\phi \in \Phi$, $M_{t}(\phi)$ is a real-valued martingale, then it has a c\`{a}dl\`{a}g version. Then it only remains to verify that for every $T>0$ the family $\{ M_{t}: t \in [0,T]\} \subseteq  \mathcal{L}(\Phi,L^{0} \ProbSpace)$ is equicontinuous. 

Let $T>0$. First note that because for all $\phi \in \Phi$, $\{ \abs{M_{t}(\phi)}\}_{t \geq 0}$ is a positive submartingale, then by Doob's tail martingale inequality we have for any $\epsilon>0$, 
\begin{equation} \label{martingTailInequality}
\Prob \left( \sup_{t \in [0,T]} \abs{M_{t}(\phi)} > \epsilon \right) \leq \frac{1}{\epsilon} \Exp ( \abs{M_{T}(\phi)}). 
\end{equation} 
Define $p_{T}: \Phi \rightarrow [0,\infty)$ by $p_{T}(\phi)=  \Exp ( \abs{M_{T}(\phi)} )$, for all $\phi \in \Phi$. Then, $p_{T}$ is a semi-norm on $\Phi$. Moreover, from the continuity of the map $M_{T}: \Phi \rightarrow L^{0} \ProbSpace$ and 
by following similar ideas to those used in the proof of Lemma \ref{lemmContiXtOnCHT}, we can show that there exists a countably Hilbertian topology $\theta$ on $\Phi$ such that the map $M_{T}$ is $\theta$-continuous. By means of an application of the closed graph theorem as in the proof of Lemma \ref{lemmContiXAsMapToSpaceContiProcesses} we can show that $M_{T}$ is continuous as a map from $\Phi_{\theta}$ into $L^{1} \ProbSpace$. This in turns implies that $p_{T}$ is continuous. Then, it follows from \eqref{martingTailInequality} that the family $\{ M_{t}: t \in [0,T]\} \subseteq  \mathcal{L}(\Phi,L^{0} \ProbSpace)$ is equicontinuous. Therefore, the existence of the $\Phi'_{\beta}$-valued \cadlag version $\widetilde{M}$ of $M$ is a consequence of Theorem \ref{theoRegularizationTheoremCadlagContinuousVersion}. 

If $M$ is $n$-th integrable, for $n \geq 2$, then it follows from Doob's inequality that for every $T>0$ and $\phi \in \Phi$ we have 
$$ \Exp \left( \sup_{t \in [0,T]} \abs{M_{t}(\phi)}^{n} \right) \leq m^{n} \, \Exp ( \abs{M_{T}(\phi)}^{n}) < \infty, $$
where $1/n+1/m=1$. Then, \emph{(1)} follows from Theorem \ref{theoExistenceCadlagContVersionHilbertSpaceFiniteMoments}. 
%

To prove \emph{(2)}, note that for each $\phi \in \Phi$ the condition $\sup_{t \geq 0} \Exp \left( \abs{M_{t}(\phi)}^{n} \right)< \infty$ implies that $M(\phi)=\{ M_{t}(\phi) \}_{t \geq 0} \in \mathcal{M}^{n}(\R)$. Then, $M$ defines a linear map $\phi \mapsto M(\phi)$ from $\Phi$ into $\mathcal{M}^{n}(\R)$. We can show that this map is continuous by using similar arguments to those used in Lemma \ref{lemmContiXAsMapToSpaceContiProcesses}. Then, it follows from \eqref{defiSemiNormSpaceBoundedMartingales} that $p:\Phi \rightarrow [0,\infty)$ defined by 
$$ p(\phi)= \left[ \Exp \left( \sup_{t \geq 0} \abs{M_{t}[\phi]}^{n} \right) \right]^{1/n}, \quad \forall \, \phi \in \Phi, $$
is a continuous semi-norm on $\Phi$. Let $\varrho$ be a continuous Hilbertian semi-norm on $\Phi$ such that $p \leq \varrho$. Then, it follows from the definition of $p$ and Chebyshev's inequality that for every $\epsilon >0$ we have 
$$ \Prob \left( \sup_{t \geq 0} \abs{M_{t}(\phi)} > \epsilon \right) \leq \frac{1}{\epsilon^{n}} p(\phi)^{n} \leq \frac{1}{\epsilon^{n}} \varrho(\phi)^{n}, \quad \forall \, \phi \in \Phi.$$ 
Therefore, for every $t \geq 0$ the map $M_{t}:\Phi \rightarrow L^{0} \ProbSpace$ is $\varrho$-continuous. Hence, Theorem \ref{theoRegularTheoHilbertSpaceCadlagContinuousVersion} shows that 
there exists a continuous Hilbertian semi-norm $r$ on $\Phi$, $\varrho \leq r$ such that $i_{r,\varrho}$ is Hilbert-Schmidt and such that $M$ has a $\Phi'_{r}$-valued c\`{a}dl\`{a}g version $\widehat{M}=\{ \widehat{M}_{t} \}_{t\geq 0}$ that is a martingale. Moreover, because $p \leq r$ is clear that $\widehat{M}$ defines a continuous and linear map $\phi \mapsto \widehat{M}(\phi)$ from $\Phi_{r}$ into $\mathcal{M}^{n}(\R)$. Then, from a modification of the arguments in Step 2 of the proof of Theorem \ref{theoExistenceCadlagContVersionHilbertSpaceFiniteMoments}, we can show that there exists a continuous Hilbertian semi-norm $q$ on $\Phi$, $r \leq q$, such that $M$ has a $\Phi'_{q}$-valued c\`{a}dl\`{a}g version $\widetilde{M}=\{ \widetilde{M}_{t} \}_{t\geq 0}$ that is a martingale and satisfies $ \Exp \left( \sup_{t \geq 0}  q'( \widetilde{M}_{t} )^{n} \right) < \infty$. 
\end{prf}

\begin{rema}
The results in Theorem \ref{propertiesMartingales} (for the case $n=2$) were originally proved by Mitoma (see \cite{Mitoma:1981}; see also \cite{KallianpurXiong}, Theorems 3.1.3-4) for the case of martingales in the strong dual of a nuclear Fr\'{e}chet space. Note that we have been able to extend these results to any nuclear space and for $n$-integrable martingales for any $n \geq 2$.   
\end{rema}

\textbf{Acknowledgements} I want to express my deepest gratitude to my PhD Supervisor David B. Applebaum for all his advice, comments and suggestions that contributed to improve this work. Also, I would like to thank both the School of Mathematics and Statistics (SoMaS) of The University of Sheffield and the Office of International Affairs and External Cooperation (OAICE) of The University of Costa Rica for providing financial support.

\end{document}